\documentclass[12pt,twoside]{article}
\usepackage{amsmath,amsthm}
\usepackage{amssymb,latexsym}
\usepackage{mathrsfs}
\usepackage{amscd}
\usepackage{enumerate}
\usepackage{indentfirst}
\usepackage{tikz}
\usepackage{amssymb}
\usepackage{xcolor}

\newtheorem{theorem}{Theorem}[section]

\newtheorem{lemma}{Lemma}[section]
\newtheorem{proposition}{Proposition}[section]

\newtheorem{remark}{Remark}[section]

\theoremstyle{definition} \theoremstyle{remark}
\numberwithin{equation}{section}

\def\B{\mathcal{B}}
\def\no{\noindent} \def\p{\partial} 
\def \Vh0{\stackrel{\circ}{V}_h} \def\to{\rightarrow}

\newcommand{\q}{\quad}    \def\R{{\mathbb R}}
  \def\f{\frac}  

\def\p{\partial}

   \def\A{{\mathcal A}}

\newcommand{\ec}
{\mathrel{\raise2pt\hbox{${\mathop=\limits_{\raise1pt\hbox{\mbox{$\sim$}}}}$}}}

\def\beq{\begin{equation}} \def\eeq{\end{equation}}

\def\beqn{\begin{eqnarray}}  \def\eqn{\end{eqnarray}}

\def\beqnx{\begin{eqnarray*}} \def\eqnx{\end{eqnarray*}}

\def\be{\begin{enumerate}} \def\ee{\end{enumerate}}
\def\bn{\begin{enumerate}} \def\en{\end{enumerate}}

\def\bd{\begin{description}} \def\ed{\end{description}}
\def\bc{\begin{cases}} \def\ec{\end{cases}}
\def\bt{\begin{tabular}} \def\et{\end{tabular}}
\def\bm{\begin{pmatrix}} \def\em{\end{pmatrix}}

\begin{document}

\title{{\bf  Variational source condition for the reconstruction  of distributed fluxes}}

\author{ De-Han Chen\footnote{School of Mathematics \& Statistics, Central China Normal University, Wuhan 430079, People'Âs Republic of China. The work of DC is supported by  National Natural Science Foundation of China (No.11701205), and also by the Alexander von Humboldt foundation through a postdoctoral research fellowship and hosted by UniversitaÃÂÃÂt Duisburg-Essen.
  (dehan.chen@uni-due.de).
}
\and
Yousept Irwin\footnote{FakultaÃÂÃÂt fuÃÂÃÂr Mathematik, UniversitaÃÂÃÂt Duisburg-Essen, Thea-Leymann-Str. 9, D-45127 Essen, Germany. (irwin.yousept@uni-due.de). The work of IY is supported by the German Research Foundation Priority Program DFG
SPP 1962 ÃÂNon-smooth and Complementarity-based Distributed Parameter Systems: Simulation and Hierarchical OptimizationÃÂ, Project YO 159/2-1.} 
\and Jun Zou\footnote{Department of Mathematics, Chinese University of Hong Kong, Shatin, N.T., Hong Kong.
(zou@math.cuhk.edu.hk).
}
}

\maketitle

%%%%%%%%%%%%%%%%%%%%%%%%%%%%%%%%%%%%%%
%%%%%%%%%Introduction%%%%%%%%%%%%%%%%%
%%%%%%%%%%%%%%%%%%%%%%%%%%%%%%%%%%%%%%%

\begin{abstract}
This paper is devoted to  the inverse problem of recovering the unknown distributed flux  on an inaccessible part of boundary using measurement data on the accessible part.  We establish and verify a variational source condition for this inverse problem, leading to a  logarithmic-type convergence rate for the corresponding Tikhonov regularization method under a  low Sobolev regularity assumption on the distributed flux. Our proof is based on the conditional stability and Carleman estimates together with  the complex interpolation theory on a proper Gelfand triple. 

\end{abstract}


\medskip\medskip
\noindent\textbf{AMS subject classification 2000}: {}

\section{Introduction}

In this paper, we analyze an inverse problem of reconstructing the distributed flux on the inaccessible part of the boundary from measurement data on the accessible part of the boundary.  More precisely, we consider the following elliptic diffusion system:
\begin{equation} \label{elliptic}
\left\{\begin{array}{lll} -\nabla\cdot(\alpha(x)\nabla u(x))=f(x)\,\, &\text{in}\,\,
\Omega, \\[1.1ex]
-\alpha(x)\displaystyle\frac{\partial u}{\partial n}(x)=k(x)(u(x)-u_a(x))
\,\,&\text{on}\,\, \Gamma_a,
\\[1.1ex]
-\alpha(x)\displaystyle\frac{\partial u}{\partial n}(x)=q(x)\,\,&\text{on}\,\, 
\Gamma_i,
\end{array}\right.
\end{equation}
where the given data include the source term $f$, the ambient
concentration or temperature $u_a$, the concentration or heat
transfer coefficient $k$ and the
diffusivity coefficient $\alpha$. The boundary is given by   $\partial \Omega=\Gamma_a\cup\Gamma_i$, where $\Gamma_a$  denotes the accessible part, and
$\Gamma_i$ is the inaccessible one. 
The Neumann boundary term $-\alpha(x){\partial u}/{\partial n}=q(x)$ on
$\Gamma_i$, referred to as
{\it the distributed flux}, is the main concern of this work:

\bigskip
\noindent\textbf{(IP)} Given   noisy data $u^\delta$  of  the exact solution
$u^\dag$ on the accessible part $\Gamma_a$, we aim at
recovering the   distributed flux $q^\dag$ on the
inaccessible part $\Gamma_i$.
\bigskip

This inverse problem finds numerous important applications in diffusive, thermal and heat transfer problems, including the
real-time monitoring in steel industry \cite{Ali} and the
visualization by liquid crystal thermography \cite{DK}. Since it is  difficult to obtain an accurate measurement on the inaccessible boundary, such as the interior boundary of nuclear reactors and steel furnaces, engineers attempt to reconstruct the flux from the measurements on the accessible boundary. This leads to the inverse problem {\bf(IP)}, which is a severly ill-posed Cauchy problem in Hadamard's sense \cite{Had}:  If we replace the exact data  $u^\dag$ on  
the accessible part $\Gamma_a$ by a noisy pattern $u^\delta$ satisfying 
\beq\label{eq:noisy}
\|u^\dag-u^\delta \|_{\Gamma_a}\leq \delta, 
\eeq
where $\delta>0$ represents the noisy level, then there may not exist a solution to {\bf (IP)}. Even if a   solution exists and $\delta>0$ is very small, it may be far away from the exact one $q^\dag$. 
We refer to \cite{LiXieZou11,xiezou,ZK} and the references therein for theoretical and numerical results related to  {\bf (IP)}.

To deal with the ill-posedness, we  shall consider the   Tikhonov regularization technique by solving a  least-squares   minimization problem. Our main goal is to examine  the convergence rate of the regularized solution under an appropriate  choice of the noise level $\delta>0$ and the Tikhonov regularization parameter. It is well-known that a  smoothness assumption  on the true solution ({\it source condition}),  is required to obtain the convergence rate. In general, a (classical) source condition requires  the Fr\'{e}chet differentiability of the forward operator and further properties on the adjoint of the Fr\'{e}chet-derivative (cf. \cite{Lechleiter08,Kaltenbacher09}).   Our present work   focuses on the  so-called {\it variational source condition} (VSC).   The concept of VSC   was  originally introduced by Hofmann et al. \cite{HKPS07} based on the use of  a linear index function.  Convergence rates for a more general index function  were  proven independently in  \cite{BH10}, \cite{Fle10} and \cite{Grasm10}.  The paper \cite{Flemmingbuch12} contains a  modified proof of  the convergence rate result by \cite{Grasm10}.
Compared to the classical source condition,  
VSC does not require any differentiability assumption on the forward operator. More importantly,  convergence rates for the regularized solutions follow immediately from VSC under an appropriate parameter choice rule (see \cite{HofMat12}).

To the best of our  knowledge,  there are only very few contributions towards VSC  for  inverse problems governed by partial differential equations. 
% In particular, although there are some theoretical investigations and finite element simulation of the Tikhonov regularization for inverse heat flux problems \cite{LiXieZou11,xiezou}, there is few paper on the validity of  VSC for this problem. 
Hohage and  Weidling \cite{Hohage15,Hohage1502} derived VSC   for the Tikhonov regularization of inverse scattering problems, leading to the strong convergence with logarithmic-type rates for the corresponding regularized solutions.   VSC for ill-posed backward   Maxwell's equations was  analyzed in \cite{ChenYousept} (cf. also \cite{Yousept2013,Yousept2017} concerning the optimal control of nonlinear Maxwell's equations). For abstract linear operators, we   refer to \cite{AHR13,BFH13,ChenHofZou16,Hohage16} and the references therein.   A general criterion for the verification of VSC for linear inverse problem was established in \cite{Hohage16}.   For more details between VSC and classical source conditions, we refer the reader to \cite{Hohage15,Hohage1502,Hohage16}.  See also  \cite{ChenHofZou16} concerning recent results on VSC for elastic-net regularizations.

The main purpose of this paper is to establish and verify VSC  for the Tikhonov regularization of the inverse problem {\bf (IP)}. To this end, we will first establish a   sufficient condition on VSC for  a general ill-posed problem  (Lemma \ref{lemma:vsc}), 
 in terms of a sequence of appproximating orthogonal projectors, which is an extension to the one introduced in \cite{Hohage16}. The proposed sufficient condition consists of two separate conditions characterizing  the smoothness of the exact solution  and the ill-posedness of the inverse of the forward problem, respectively. 
In order to apply the sufficient condition to the inverse problem {\bf (IP)}, we shall  derive a   {\it conditional stability estimate} (Theorem \ref{continuity})  for every function $u\in H^2(\Omega)$ satisfying a specific second-order elliptic equation \eqref{eq:homo}.  In particular, the proposed  estimate reveals the dependence of   $\|u\|_{H^1(\Omega)}$ on  
$\|u\|_{L^2(\Gamma_a)}$ under  an a priori bound on the $H^2(\Omega)$-norm. The main tools to prove  the conditional stability estimate are the developed techniques    by \cite{Bourgeois,Phung} and  a specialized Carleman estimate (Lemma \ref{lemma:carleman}). Eventually, our result extends \cite{Bourgeois}, as the derived estimate makes use of  the $L^2(\Gamma_a)$-norm  of the Dirichlet data  (Proposition \ref{pro:boundary-boundary}  and Theorem \ref{the:H1:log}) instead of   the $H^1(\Gamma_a)$-norm as   in  \cite[Propositions 2.2-2.4 \& Theorems 2.2-2.3]{Bourgeois}.   
Finally, based on  the  developed conditional stability estimate in combination with the complex interpolation theory and the   Gelfand triple  $H^{1/2}(\Gamma_i)\subset L^2(\Gamma_i)\subset H^{-1/2}(\Gamma_i)$,
we   are able to establish a sequence of appproximating orthogonal projectors and 
 prove our main result on VSC (Theorem \ref{the:main}). In particular, it leads to a logarithmic-type convergence rate for the Tikhonov regularization under a low   Sobolev regularity assumption on the exact distributed flux $q^\dag\in H^s(\Gamma_i)$ with $s\in (0,1/2]$.

This paper is organized as follows.  Section \ref{sec:math} provides  the precise mathematical formulation of the inverse problem {\bf (IP)} and states our main theoretical findings (Theorems \ref{the:main} and \ref{continuity}). In Section \ref{sec:abstract}, we   present  a sufficient condition on VSC for a general ill-posed problem. The proof of  the conditional stability estimate (Theorem \ref{continuity})  is provided in  Section \ref{sec:aux}. The final section is devoted to the derivation of  Theorem  \ref{the:main}  on  VSC for  
  {\bf (IP)}.

\section{Mathematical formulation and main results}\label{sec:math}

We begin this section by recalling some terminologies and notations used in the sequel.  Given  a linear  operator $T:X \to X$ on a complex Hilbert space $X$, the notations $D(T)$, $\rho(T)$ and $\sigma(T)$ stand for the domain, resolvent and spectrum of $T$, respectively. A linear  operator 
$T:D(T)\subset X\to X$ is called closed, if its graph $\{(x,Tx), ~ x \in D(T)\}$ is closed in  $X\times X$. Furthermore, the adjoint  of  a densely defined operator $T:D(T)\subset X\to X$  is denoted by $T^*:D(T^*)\subset X\to X$. We call  $T:D(T)\subset X\to X$ symmetric,  if $Tx=T^*x$ holds true for all $x\in D(T)$, i.e., 
$(Tx,y)_X =( x,Ty)_X $ for all $x,y\in D(T)$. If a symmetric operator $T$ satisfies that $D(T)=D(T^*)$, then   $T$ is said to be self-adjoint.

  For $1\leq i,j\leq d$ and a     function $u$ defined on $\R^d$, we write $\p_i u:=\p u/\p x_i$, $\p_{i,j}u:=\p^2 u/\p x_i \p x_j$, and $\nabla u =(\p_1 u,\ldots,\p_d u)$.  Given the Hessian matrix $u''$ of a function $u$, we write $u''({x},{y}):=\sum_{i,j=1}^d \p_{i,j}u \cdot x^i\cdot y^j$ with the vectors ${x}=(x^1,\cdots,x^d), {y}=(y^1,\cdots,y^d) \in \mathbb{C}^d$. 
In addition, we will often use the notation $C$ to denote
generic positive constant  independent of the parameter or 
functions involved. Also, we use the expression 
$A\lesssim B$ to indicate that $A\leq C B$ for 
a positive constant $C$ that is independent of $A$ and $B$. 
For two Banach spaces $X$ and $Y$  that are continuously embedded in the same Hausdorff topological vector space, we denote by $[X,Y]_\theta$ ($0\leq\theta\leq 1$)   the complex interpolation space between $X$ and $Y$. 

For every $-\infty< s<\infty$, we define  fractional Sobolev space 
$$
H^s(\R^d):=\{u\in \mathcal{S}(\R^d)'\mid \|u\|^2_{H^s(\R^d)}:=\int_{\R^d}(1+|\xi|^2)^{s}|(\mathcal{F}u)(\xi)|^2 d\xi<+\infty  \},
$$
where $\mathcal{F}:\mathcal{S}(\R^d)'\to \mathcal{S}(\R^d)'$ is the Fourier transform and $\mathcal{S}(\R^d)'$ denotes the tempted distribution space  (see, e.g.,  \cite{Lions,Wloka,Yagibook}). For a bounded domain $ U \subset \R^d$ with a Lipschitz boundary $\p U$,  the space  $H^s(U)$ with a possibly non-integer exponent $s\geq 0$ is defined as the space of all complex-valued functions $v\in L^2(U)$ satisfying 
$V_{\vert U}=v$ for some $V\in H^s(\R^n)$, endowed with the norm 
$$
\|v\|_{s,U}:=\inf_{\substack{V_{\vert U}=v \\ V\in H^s(\R^n)}}\|V\|_{H^s(\R^n)}.
$$
{ When no confusion may be caused,
we   simply drop $U$ in the subscription of  $\|\cdot\|_{s,U}$.}  
For every $s \in [0,\infty)$, we denote by $\lfloor s \rfloor \in [0,s]$ the largest integer less or equal to $s$. In the case of $s \in (0,\infty)$ with $s= \lfloor s \rfloor + \sigma$ and $0<\sigma<1$, the norm $\|\cdot\|_{s,U}$ is equivalent  to  (cf. \cite{Yagibook})
$$
 \left(\sum_{|\alpha|\leq \lfloor s \rfloor } \|D^\alpha u\|_{ L^2(U) }^2+\sum_{|\alpha|\leq \lfloor s \rfloor } \, \, \iint\limits_{U\times U}\f{|D^\alpha u(x)-D^\alpha u(y)|^2}{|x-y|^{n+2\sigma}} dxdy\right)^{\f{1}{2}}.
$$
 If $s$ is a non-negative integer, then  $H^s(U)$  coincides with the classical Sobolev space. For a compact, $d$-dimensional   $C^{k,\kappa}$-manifold $M$ with an integer $k\geq 0$ and 
 $\kappa\in \{0,1\}$, we can define the Sobolev space $H^{s}(M)$ on $M$
 for all $0\leq s\leq k+\kappa$ via partitions of unity and Sobolev spaces $H^s(\R^d)$ (see, e.g.,  \cite{Wloka}). 
   { More precisely,  let $M$ be a $d$-dimensional $C^{k,\kappa}$ compact manifold with an admissible 
   $C^{k,\kappa}$-atlas $\{(U_i,\alpha_i)\}_{i\in I}$ for $M$ and a subordinate partition of unity $\beta_i$. Since $M$ is compact, the indexing set $I$ can be chosen to be finite, and possibly  by shrinking $U_i$,  the mappings $\alpha_i\circ \alpha_j^{-1}$ are  $\widetilde{ C}^{k,\kappa}$-diffeomorphisms  (cf. \cite[Proposition 4.1 ]{Wloka}).  Then, for $0\leq s\leq k+\kappa$, $H^s(M)$  is specified by  the space  of all $u\in L^2(M)$   such that for all $i\in I$ the functions
   $$
(u\cdot \beta_i)\circ \alpha_i^{-1}: \alpha_i(U_i)\subset \R^d\to \mathbb{C}  
$$  
 belong to $H^s(\R^{d})$. This forms   a Hilbert space equipped with 
 $$
 (u,v)_{H^s(M)}=\sum_{i\in I}((u\cdot \beta_i)\circ \alpha_i^{-1},(v\cdot \beta_i)\circ \alpha_i^{-1} )_{H^s(\R^{d})} \quad \textrm{and} \quad \|u\|_{s,M}:= \sqrt{(u,u)_{H^s(M)}}.
 $$
 One can show further that this definition is independent of the particular choice of atlas and partition of unity.  
 In particular, for a bounded  domain $U$ of class  $C^{k,\kappa}$, its boundary $\p U$ is a compact $C^{k,\kappa}$ manifold and  then the Sobolev space $H^s(\p U)$ is defined as above. By $H^{-s}(M)$ we denote the dual of $H^s(M)$ with respect to the inner product in $L^2(M)$.  We also write $H^0(M)=L^2(M)$ and denote its norm and scalar product by $\|\cdot\|_M$ and $(\cdot,\cdot)_M$. By a standard argument (as used in \cite[Theorem 7.7]{Lions}), for every $-(k+\kappa) \leq s_1<s_2\leq k+\kappa$, one has for all $\theta\in [0,1]$, 
\beq\label{comp:interp}
[H^{s_1}(M), H^{s_2}(M)]_\theta = H^{s_1(1-\theta)+s_2\theta}(M),
\eeq
with equivalent norms.   }

% In addition, if $k\geq 1$, then the boundary is of $N^{k,\kappa}$ property. In this case,  for  $s= \lfloor s \rfloor + \sigma$ and $0<\sigma<1$,  the norm of $H^s(\p U; \mathbb{C})$  can be characterized by surface integral 
%$$
%\|u\|_{H^s(\p U \mathbb{C})}^2
%=\sum_{|\alpha|\leq \lfloor s \rfloor } \|D^\alpha u\|_{ L^2(\p U;\mathbb{C}) }^2+\sum_{|\alpha|\leq \lfloor s \rfloor } \, \, \iint\limits_{\p U\times\p U }\f{|D^\alpha u(x)-D^\alpha u(y)|^2}{|x-y|^{n-1+2\sigma}} ds ds'. 
%$$ 

Let us now formulate  the   general    assumptions on the domain $\Omega$ and the coefficients involved  in \eqref{elliptic}.

\begin{enumerate}

\item[\bf{(H1)}] The domain $\Omega \subset \R^d$,  $d\geq 2$, is connected,  bounded and  of class  $C^{1,1}$. There exist  ($d-$1)-dimensional  compact $C^{1,1}$-manifolds  $\Gamma_a,\Gamma_i \subset  \R^{d}$ such that  $\Gamma_a \cap \Gamma_i = \emptyset$ and $\partial \Omega = \Gamma_i \cup \Gamma_a$.

\item[\bf{(H2)}]
In addition, $\alpha\in C^{1,1}(\overline{\Omega})$, $k\in 
C^1(\Gamma_a)$ with 
\beq
\alpha_{min}:=\min_{x\in \overline \Omega}{\alpha(x)}>0,~k_{min}:=\min_{x\in\Gamma_a}k(x)>0,
\eeq 
$f\in L^2(\Omega)$ and $u_a\in H^{1/2}(\Gamma_a)$.  
\end{enumerate}
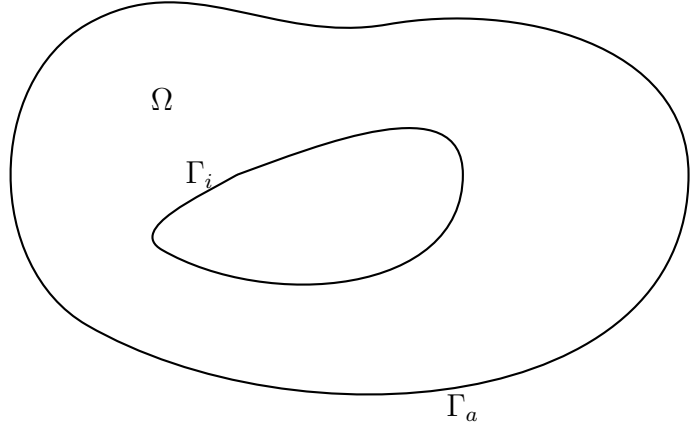
\begin{figure}[h!]
  \begin{center}
  \begin{tikzpicture}
  %\draw[help lines] (-4,-4) grid (6,5);
   \draw [thick] (-2,2) % Draws a line
      to [out=30,in=190] (2,2) % out, put is the angle control
      to [out=10,in=90] (6,0) 
      to [out=-90,in=-30] (-2,-2)
      to  [out=150,in=-150] (-2,2);
      \draw [thick] (0,0) % Draws a line
      to [out=30,in=180] (0,0) % out, put is the angle control
      to [out=20,in=90] (3,0) 
      to [out=-90,in=-30] (-1,-1)
      to  [out=150,in=-150] (0,0);
       \draw (-1,1) node {$\Omega$};% add text and symbols
            \draw (-0.5,0) node {$\Gamma_i$};  
            \draw (3,-3.1) node {$\Gamma_a$};
        \end{tikzpicture}
  \caption{An example of $\Omega$ in $\R^2$}
  \end{center}
\end{figure}

In particular, by {\bf{(H1)}}, the  trace operator $\text{tr}: H^{1}(\Omega) \to H^{ \f{1}{2}}(\p \Omega)$ 
 and   the   (outward) normal derivative trace operator $\f{\p}{\p n} : H^2(\Omega)\to H^{\f{1}{2}}(\p\Omega)$ are well-defined as  linear and bounded operators  (see e.g. \cite{Grisvard,Wloka}). For simplicity, we also  use the abbreviation   $\f{\p}{\p n}=\p_n$. 
%\begin{remark}\label{remark:geometric}
% If there exists a $C^{2,1}$-diffeomorphism $G:\overline{\Omega}\mapsto\{x\in\R^d; 1\leq |x|\leq R\}$ with
%$\Gamma_i=G^{-1}(\S_1^{d-1})$ and $\Gamma_a=G^{-1}(\S_R^{d-1})$, then  we can infer that there exists a 
%function $\psi$  satisfying the requirements in {\bf (H2)}. 
%
%
%Roughly speaking, the existence of $\psi$ is equivalent to the existence of a  $C^{2,1}$-diffeomorphism $G:\overline{\Omega}\mapsto M\times [1,R]$ for a  
%$d-1$ dimensional embedded $C^{2,1}$-manifold (without boundary) $M$ in $\R^d$.  
% Since the proof of this equivalence requires results from Morse theory, and goes beyond the scope of this paper, we omit it and refer the reader to 
%\cite[Theorem 2.31]{Matsumoto}.
% 
%\end{remark}

We now recall the basic   variational approach for a general inverse problem:
\beq\label{eq:general}
F(q)=g,
\eeq
 where $F:D(F)\subset X\to Y$,  with domain $D(F)$, is an operator between two
Hilbert spaces $X$ and $Y$    (see e.g. \cite{Flemmingbuch12,Had,Scherzer09}). In general,   $F^{-1}$ is unbounded so that the operator equation (\ref{eq:general}) is ill-posed.
% in the sense that 
%\eqref{eq:general} has no solution, if  the true observation data $F(q^\dag)$ comes with noisy, namely only the deterministic perturbation $g^\delta$ of true observation is available, i.e.,  
%$$
%\|F(q^\dag)-g^\delta\|_Y\leq \delta,
%$$
%where $\delta>0$ represents the noisy level. Furthermore, even if a solution exists for a very small $\delta$,  then it   may be far away from the true solution $q^\dag$. 
 To tackle the ill-posedness, we consider the Tikhonov regularization method: For a given perturbation $g^\delta$ of the exact data $F(q^\dagger)$ and a regularization parameter $\alpha>0$, we look for   the minimizer $q^\delta_\alpha \in D(F)$ of \beq\label{eq:Tik}
  \min\limits_{q\in D(F)}\left[\f{1}{\alpha}\|F(q)-g^\delta\|^2_Y+\f{1}{2}\|q\|^2_X\right].
\eeq
In general,   the convergence rate for $\|q^\dag-q^\delta_\alpha\|$ as $\delta,\alpha \to 0$ may be arbitrary slow   (see   \cite{Engl}). To  achieve a convergence rate of for the regularized solutions $\{q^\delta_\alpha\}$, a source condition on the true solution $q^\dag$ is required (see \cite{Flemmingbuch12,Had,Scherzer09}). 
As pointed out in the introduction, we focus on the variational source condition (VSC) of the form
\beq\label{vsc0}
\f{\beta}{2}\|q-q^\dag \|_X^2\leq  \f{1}{2}\|q\|_X^2-\f{1}{2}\|q^\dag\|_X^2+\Psi(\|F(q^\dag)-F(q)\|_Y) \textup{ for all } q\in D(F),
\eeq
where $\beta\in(0,1]$ and $\Psi$ is an {\it index function}, that is a   continuous and strictly increasing function $\Psi:(0,\infty)\to (0,\infty)$ satisfying $\lim\limits_{t\to 0^+} \Psi(t)=0$. If we are able to prove the existence of an index function $\Psi:(0,\infty)\to (0,\infty)$ and $\beta\in(0,1]$ satisfying \eqref{vsc0}, then the following convergence rate
$$
\|q^\delta_{\alpha(\delta)}-q^\dag\|^2_X=\Psi(\delta)\to 0 \,\,\text{as}\,\, \delta\to 0^+ 
$$
is obtained under an appropriate parameter choice on  $\delta$ and $\alpha(\delta)$
(see e.g. \cite{ChenHofZou16}). 

Let us now discuss  the Tikhonov regularization method for {\bf (IP)}. To this aim,  we first introduce the solution operator 
$$
S:L^2(\Gamma_i)\to H^1(\Omega), 
$$
that assigns to every element $q\in L^2(\Gamma_i)$  the unique solution  $u\in H^1(\Omega)$ of  the weak formulation to \eqref{elliptic}:  
$$
\int_\Omega \alpha \nabla u \cdot {\nabla \overline v} dx+\int_{\Gamma_a}ku\overline{v} dS
=\int_{\Omega} f\overline{v}dx-\int_{\Gamma_i} q\overline{v}dS-\int_{\Gamma_a}ku_a\overline{v}dS  \quad\forall\, v\in H^1(\Omega),
$$
where for simplicity we  set  $v = \text{tr} (v)$ on $\partial \Omega$ to express boundary values of a Sobolev function $v \in H^1(\Omega)$.
Furthermore, we introduce 
\begin{align*}
A:L^2(\Gamma_i) \to L^2(\Gamma_a), \quad
A(q)=\text{tr}(Sq)\mid_{\Gamma_a}.
\end{align*}
Then,  the inverse problem {\bf (IP)} is equivalent to solve the operator equation 
\beq\label{eq:F}
A(q)=u^\dag.
\eeq
The operator $A:L^2(\Gamma_i) \to L^2(\Gamma_a)$ is compact due to  the compactness of the embedding  $H^{1/2}(\Gamma_a) \hookrightarrow L^2(\Gamma_a) $. Therefore, by a well-known argument, the inverse operator equation \eqref{eq:F} is ill-posed (see also  \cite[Theorem~2.2]{xiezou} for  the parabolic cases). To tackle with the ill-posedness, we consider the Tikhonov regularization method:
\begin{equation}\label{leastsquare1}
\min\limits_{q\in \mathcal{U}_{q^\dag} }\left[\f{1}{\alpha}\|A(q)-u^\delta\|_{\Gamma_a}^2+\f{1}{2}\|q\|_{\Gamma_i}^2\right],
\end{equation}
where $\mathcal{U}_{q^\dag}$ is a non-empty, convex   and closed subset of $L^2(\Gamma_i)$ and $u^\delta$ is a noisy pattern of the exact data $u^\dag=A(q^\dag)$, i.e.,$\|u^\delta-A(q^\dag)\|_{0,\Gamma_i}\leq \delta$. By classical arguments, this quadratic minimization problem admits a unique solution $q^{\delta}_\alpha\in  \mathcal{U}_{q^\dag}$.   In this paper,   the    admissible set  $\mathcal{U}_{q^\dag}$ is specified as 
$$
\mathcal{U}_{q^\dag}:=\{q\in L^2(\Gamma_i) \, \, | \, \, \|q-{q^\dag}\|_{1/2,\Gamma_i}\leq M_0\},
 $$
with a given positive constant $M_0>0$. 
%   In view of {\bf{(H1)}} and {\bf{(H2)}},  if $q \in L^2(\Gamma_i)$, then the weak solution to \eqref{elliptic} enjoys  a higher regularity property in $H^{3/2}(\Omega)$. 
%In other words,
%\begin{equation}
%Sq \in H^{3/2}(\Omega) \quad \forall q \in \mathcal{U}_{q^\dag}.
%\end{equation}
%This follows from the classical elliptic regularity result  \cite{Lions}.  
Let us state the main result  of this paper:

%  Let us now state the main assumption on the exact  distributed flux we consider throughout this paper:
%
%\begin{enumerate}
%
%\item[\bf{(H3)}] Let  $0\neq q^\dag\in \mathcal{U}$ such that $u^\dag:=A(q^\dag)$. 
%
%
%\end{enumerate} 
%
%\begin{remark}
%If $q^\dag=0$, then  VSC is valid for every  concave index function $\Psi$. Therefore, we do not consider this trivial case.
%\end{remark}
%
%
%

 \begin{theorem}\label{the:main}
 Let  ${\bf (H1)}-{\bf (H2)}$. Suppose that $q^\dag\in H^s(\Gamma_i)$ for some $s\in  (0,1/2]$. Then,  for every $\kappa \in (0,1)$,  there exists a concave index function   $\Psi:(0,\infty)\to (0,\infty)$ satisfying  the variational source condition 
\beq\label{vscA}
\f{1}{4}\|q-q^\dag \|_{\Gamma_i}^2\leq  \f{1}{2}\|q\|_{\Gamma_i}^2-\f{1}{2}\|q^\dag\|_{\Gamma_i}^2+\Psi(\|A(q^\dag)-A(q)\|_{\Gamma_a}) \quad\forall\,q\in \mathcal{U}_{q^\dag}
\eeq
and   the decay rate condition 
\beq\label{Psi}
\displaystyle \Psi(\delta)\lesssim 
\f{1}{(\log\f{4C}{\delta})^{\f{4 s\kappa }{(1+2s)}}}~~as~~\delta\to 0^+,
\eeq
for some positive constant $C>0$. 
\end{theorem}
\begin{remark}
As mentioned above, this result implies that the regularized solution $\{q^\delta_\alpha\}$ of \eqref{leastsquare1} 
 possess the following convergence rate 
$$
\|q^\delta_\alpha-q^\dag\|^2_{\Gamma_i}\lesssim \f{1}{(\log\f{4C}{\delta})^{\f{4 s\kappa }{(1+2s)}}}\quad \delta\to 0^+,
$$
when the parameter $\alpha=\alpha(\delta,g^\delta)$ is chosen appropriately (see   \cite{ChenHofZou16,HofMat12} for more details). 
%
%The concavity assumption on the index function is only a technical assumption.  
%From \cite[Remark 1]{HofMat12}, it is easy to see that if a VSC (\ref{vsc0}) holds with an index function $\Psi$ satisfying $\lim_{\delta\to 0^+}\f{\Psi(\delta)}{\sqrt{\delta}}\nearrow+\infty$, then we can always find a concave majorant index function $\Psi^{*}$ still fulfills the same VSC. 
%In addition, if $\Psi(\delta^2)\leq \Psi_0(\delta)$ for $\delta \in (0,\delta_0)$ with some $\delta_0>0$, and  $\Psi_0$ is an index function and is  concave over $(0,\delta_0)$, then we can find a concave majorant index function $\Psi^*$ of $\Psi$ fulfilling $\Psi^{*}(\delta)\lesssim \Psi_0(\sqrt{\delta})$. We note here that {$\Psi_0(\sqrt{\delta})$ is also concave over $(0,\delta_0)$}. Therefore, in view of the fact that for any $r>0$, the function $\delta\to 1/|\log\f{4C}{\delta}|^r$ is concave over $(0,4Ce^{-r-1})$,  Theorem \ref{the:main} implies at least a logarithmic convergence rate of the regularized solution $q^\delta_\alpha$, provided that the regularization parameter $\alpha=\alpha(\delta,q^\delta)$ is chosen properly.
\end{remark}

The key tool to prove Theorem \ref{the:main} is a conditional stability estimate 
for  every function $u \in H^2(\Omega)$ satisfying
\beq\label{eq:homo}
\left\{\begin{array}{lll}
\nabla\cdot (\alpha(x) \nabla u(x))=0\quad &\textrm{in } \Omega,\\[1.1ex]
-\alpha(x) \displaystyle \f{\p u(x)}{\p n}=k(x) u(x)\quad &\textrm{on } \Gamma_a. 
\end{array}\right.
\eeq
 { We shall prove an estimate for $\|u\|_{1,\Omega}$ depending on $\|u\|_{\Gamma_a}$ under the {\it a priori} bounded set: }
\beq\label{eq:M}
\mathcal{M}_M:=\{u\in H^{2}(\Omega) \ |\|u\|_{2,\Omega} \leq M \},
\eeq
where    the constant $M>0$ is a prescribed constant.
\smallskip 
{
\begin{theorem}\label{continuity}
Assume that ${\bf (H1)}-{\bf(H2)}$ hold. 
Let   $u \in H^2(\Omega)$ be a  function satisfying (\ref{eq:homo}) within $ \mathcal{M}_M$ defined by (\ref{eq:M}). Then, for every 
$\kappa\in (0,1)$,  there exist two positive  constants $C,C_0>0$, independent of $u$ and $M$, such that 
\beq
\|u\|_{1,\Omega}\leq \f{CM}{\log\left(\f{C_0M}{\| u\|_{\Gamma_a}}\right)^\kappa}.
\eeq
 In particular, if $\| u\|_{\Gamma_a}=0$, then  $u$ vanishes. \end{theorem} \color{blue}}

%%%%%%%%%%%%%%%%%%%%%%%%%%%%%%%%%%%%%%%%%%%%
%%%%%%%%%Absract Approach%%%%%%%%%%%%%%%%%%%
%%%%%%%%%%%%%%%%%%%%%%%%%%%%%%%%%%%%%%%%%%%%%%

\section{Sufficient condition for VSC}\label{sec:abstract}

In this section, we   present a sufficient condition to verify   
VSC for \eqref{eq:general}, which is an extension   of  \cite[Theorem 2.1]{Hohage16}.     Evidently,   VSC of the form (\ref{vsc0}) is equivalent to the form below:
\beq\label{vsc1}
\Re ( q^\dag, q^\dag-q )_X\leq  \f{1-\beta}{2}\|q^\dag-q\|^2_X+\Psi(\|F(q^\dag)-F(q)\|_Y) \quad \forall\, \,q\in D(F).
\eeq
Thus, we shall verify \eqref{vsc1} instead of (\ref{vsc0}) directly.  The result below yields a sufficient condition for  \eqref{vsc1}. The  proof follows  the lines  of \cite[Theorem 2.1]{Hohage16} and \cite[Theorem 2.5]{BFH13}.  However, since our result considers the concavity of index functions as well as the complex settings,  we could not find a precise reference covering our situation.  In addition,  significant modifications to the original idea are necessary. Therefore,   we include a  proof  for the sake of  completeness.   

\begin{lemma}\label{lemma:vsc}
Let $0 \neq q^\dag \in D(F)$, $\lambda_0\geq 0$, $\{P_\lambda\}_{\lambda\geq \lambda_0} $ be a family of orthogonal projectors   from $X$ to $X$,  and let $f, g: [\lambda_0,\infty)\to \R^+$  be continuous functions satisfying   
\begin{enumerate}
\item[$\bullet$]   $f$ is strictly decreasing and fullfils  $\lim\limits_{\delta \to\infty}f(\delta)=0$ 
\item[$\bullet$]    $g$ is strictly increasing and  fullfils  $\lim\limits_{\delta \to\infty}g(\delta)=+\infty$.
\end{enumerate}
Furthermore, suppose that there exist a concave index function $\Psi_0$ and two constants $\widehat C \geq 0$ and  $\beta\in (0,1)$  such that  
\begin{align}
&\|q^\dag-P_\lambda q^\dag\|_{X}\leq  f(\lambda), {\color{red} \text{}} \label{firstconditon}\\
&\Re(q^\dag, P_\lambda (q^\dag-q))_X\leq  g(\lambda)\Psi_0(\|F(q^+)-F(q)\|_Y) + \widehat C f(\lambda)\|q-q^\dag\|_X \label{secondcondition} \\
&\qquad \qquad   \qquad   \qquad   \qquad   \qquad   \qquad   \qquad     \forall~q\in D(F)\, \text{with}\,\|q^\dag-q\|_X<\f{2}{1-\beta} \|q^\dag\|_X\nonumber 
\end{align}
holds for all $\lambda\geq \lambda_0$. Then,    \eqref{vsc1} holds true with the following 
 concave index function:
 \beq\label{eq:psi}
\Psi : (0,\infty) \to (0,\infty), \quad \Psi(t):=\inf_{\lambda\geq \lambda_0 }\left(g(\lambda)\Psi_0({t}) + \f{(\widehat C+1)^2}{2(1-\beta)}f(\lambda)^2 \right).
\eeq
This index function      satisfies  the decay estimate
\begin{equation}\label{decayesti} 
\Psi(\delta)\lesssim (f^2\circ\Theta^{-1})(\Psi_0(\delta))\,\, \text{as}\,\,\delta\to +0,
\end{equation}
where    $\Theta^{-1}$ is the inverse   of   $\Theta: [\lambda_0,\infty) \to (0,\infty), \quad \lambda\mapsto \f{(\widehat C+1)^2}{2(1-\beta)} \frac{f(\lambda)^2}{g(\lambda)}$, which obviously satisfies 
$\lim\limits_{\lambda\to +\infty}\Theta(\lambda)=0$. 
\end{lemma}

%\begin{remark}
%We need to show that the function  $\Psi$ defined by (\ref{eq:psi})
%is indeed an index function. 
%
%
%
%By the definition, $
%\Psi$ is strictly increasing and u.s.c. since for each $k$, $(C+1)^2a(k)^2+b(k)\Psi_0({t})$ is continuous and strictly increasing. Also, it follows immediately that $\Psi$  is also l.s.c., which implies that it is continuous.
%\end{remark}

\begin{proof}
For each $q$ satisfying $\|q^\dag-q\|_X\geq\f{2}{1-\beta} \|q^\dag\|_X$,    the Cauchy-Schwarz inequality implies 
$$
\Re(q^\dag, q^\dag-q)_X \leq \f{1-\beta}{2}\|q^\dag-q\|^2_X.
$$
Therefore, we  only need to show   (\ref{vsc1}) for   $\|q^\dag-q\|_X<\f{2}{1-\beta} \|q^\dag\|_X$. For this case, using the orthogonal projection $\{P_\lambda\}_{\lambda\geq \lambda_0}$ and \eqref{firstconditon}-\eqref{secondcondition}, it follows   that 
\begin{align*}
\Re(q^\dag, q^\dag-q)_X &=\Re( P_\lambda q^\dag, q^\dag-q)_X+\Re((I-P_\lambda)q^\dag, q^\dag-q)_X \\
&\leq \Re(q^\dag, P_\lambda (q^\dag-q))_X+  f(\lambda)\|q^\dag-q\|_X\\
&\leq g(\lambda)\Psi_0(\|F(q^\dag)-F(q)\|_Y)+\f{1-\beta}{2}\|q^\dag-q\|^2_X+\f{(\widehat C+1)^2}f(\lambda)^2,
\end{align*}
where we have used Young's inequality for the last inequality. In conclusion, for every $\lambda\geq \lambda_0$,  the inequality 
$$
\Re(q^\dag, q^\dag-q)_X\leq \f{1-\beta}{2}\|q^\dag-q\|_X^2+
 g(\lambda)\Psi_0(\|F(q^\dag)-F(q)\|_Y)+\f{(\widehat C+1)^2}{{ 2(1-\beta)}}f(\lambda)^2
$$
holds for all $q\in D(F)$.  Thus, defining $\Psi : (0,\infty) \to (0,\infty)$ as in \eqref{eq:psi}, we see that  \eqref{vsc1} is satisfied. 
It remains to show that $\Psi: (0,\infty) \to (0,\infty)$ is   a concave index function.  First, 
since $\Psi:(0,\infty) \to (0,\infty)$ is an infimum of  concave 
functions,  we obtain that $\Psi: (0,\infty) \to (0,\infty)$ is concave.  Since for $t\in(0,\infty)$, $-\Psi(t)>-\infty$ and $-\Psi$ is convex over $(0,\infty)$, we infer that $\Psi: (0,\infty) \to (0,\infty)$ is continuous   \cite[Corollary 47.6]{Zeidler}.  

Finally, we prove the decay estimate (\ref{decayesti}), which also implies the continuity of $\Psi$ at 0. 
Since  { $\lim\limits_{\delta\to 0^+} \Psi_0(\delta)=0$}, if $\delta$ is sufficiently small, there exits a unique $\lambda$ such that 
\beq
\Psi_0(\delta)=\f{(\widehat C+1)^2}{ 2(1-\beta)}\f{f(\lambda)^2}{g(\lambda)},
\eeq
 as the function $\lambda\mapsto \f{(\widehat C+1)^2}{ 2(1-\beta)}\f{f(\lambda)^2}{g(\lambda)} $ is continuous, strictly deceasing and convergent  to $0$ as $\lambda\to \infty$.
If we set $\Theta(\lambda):=\f{(\widehat C+1)^2}{2(1-\beta)}\f{f(\lambda)^2}{g(\lambda)}$,  then we obtain $\lambda=\Theta^{-1}(\Psi_0(\delta))$, which yields 
$$
\Psi(\delta)\leq \f{(\widehat C+1)^2}{1-\beta} f(\Theta^{-1}(\Psi_0(\delta)))^2 \quad \text{as}\,\,\delta\to +0,
$$
and consequently the decay estimate \eqref{decayesti} is obtained.

To show that  $\Psi$ is strictly increasing, we choose $t_1,t_2$ with $0<t_1<t_2$. For  $t_2>0$, it holds that  
$$
 g(\lambda)\Psi_0({t_2}) + \f{(\widehat C+1)^2}{2(1-\beta)}f(\lambda)^2\to \infty, \quad  \text{ as } \lambda\to +\infty.
 $$
  Thus, according to the definition \eqref{eq:psi}, there exist $\lambda_*$ such that $\Psi(t_2)=g(\lambda_*)\Psi_0({t_2})+ \f{(\widehat C+1)^2}{2(1-\beta)}f(\lambda_*)^2$.  Then, as $\Psi_0$ is strictly increasing, it follows that  
\beq
\Psi(t_1)\underbrace{\leq}_{\eqref{eq:psi}} \left(g(\lambda_*)\Psi_0({t_1}) + \f{(\widehat C+1)^2}{2(1-\beta)}f(\lambda_*)^2 \right) <\Psi(t_2). 
\eeq
\end{proof}
\section{Derivation of  Theorem \ref{continuity}}\label{sec:aux}

The goal of this section is to prove Theorem \ref{continuity},   which is a critical auxiliary result for the proof of  Theorem  \ref{the:main}. 
% In view of Lemma \ref{lemma:vsc}, we need to find suitable projection $\{P_\lambda\}_{\lambda\geq \lambda_0}$ such that, 
%for any $q\in \mathcal U$, the real part of inner product $\Re ( q^\dag-q,P_\lambda (q^\dag-q))_{\Gamma_i}$ can be estimate by the norm of 
%$F(q^\dag)-F(q)$.  
The importance of Theorem \ref{continuity} lies in the fact that it  establishes the continuous dependence of  
$q^\dag -q $ on 
$(u(q^\dag)-u(q))_{\mid_{\Gamma_a}}$ under a priori bound on $u(q^\dag)-u(q)$. After proving  Theorem \ref{continuity}, we shall construct suitable orthogonal projections  $\{P_\lambda\}_{\lambda\geq \lambda_0}$ such that  $P_\lambda (q^\dag-q)$ can be estimated by $ (u(q^\dag)-u(q))_{\mid_{\Gamma_i}}$.

{
To prove Theorem \ref{continuity}, we   follow and modify  the techniques by \cite{Bourgeois}.  To be more precise, we   establish a Carleman estimate and   derive a series of 
local conditional estimates results. Combining all these theoretical findings   yields the desired global result.  Let us underline that our results extend  \cite{Bourgeois}: We study  a   class of  second-order elliptic operators of divergence form, while  \cite{Bourgeois} focuses on the Laplace operator. Furthermore, our  estimates  make use of the $L^2(\Gamma_a)$-norm  of the Dirichlet data (Proposition \ref{pro:boundary-boundary}  and Theorem \ref{the:H1:log}) instead of the   $H^1(\Gamma_a)$-norm as considered in  \cite[Propositions 2.2-2.4 \& Theorems 2.2-2.3]{Bourgeois}. 
}

Let us first derive a global Carleman estimate for the following  second-oder elliptic operator:
\beq\label{eq:nonhomo}
(L_\alpha u)(x):=-\nabla\cdot(\alpha(x)\nabla u(x)).
\eeq
  Carleman estimate was initially introduced to study the quantification of unique continuation,  which goes back to the early work of T. Carleman himself. In recent years, Carleman's estimate has been applied to the fields of control theory and inverse problem for PDEs  (see e.g. \cite{Rousseau12,RLebeau12}).  
Our derivation here follows from a computational method   by A. Fursikov and O. Yu. Imanuvilov (see e.g. \cite{Fursikov,Rousseau12})

\begin{lemma}\label{lemma:carleman}
Let $U\subset \mathbb{R}^d$ be a bounded Lipschitz domain and  $\psi\in C^{1,1}(\overline{U})$  be a real-valued function such that 
$$
|\nabla \psi(x)|\geq c \,\, \text{for all}\, \,x\in \overline{U}\,\,\text{and}\,\,
D^2\psi\in L^\infty(U)^{d\times d}
$$ 
with some positive constant $c>0$. Furthermore,  let  $\varphi=e^{\gamma \psi}$ with $\gamma>0$. Then, there exist { $\gamma_{*}\ge 1$ and $\tau_* \ge 1$ such that  
\beq\label{eq:Carleman}
\begin{split}
\displaystyle \int_{U}(\tau^3\gamma^4 \varphi^3 e^{2\tau\varphi}|u|^2+\tau \gamma^2 e^{2\tau \varphi} |\nabla u|^2)dx \lesssim  \int_{U}e^{2\tau\varphi}|L_\alpha u|^2dx + \\
\displaystyle   + \int_{\p U}\left(\tau^3 \gamma^3 \varphi^3 e^{2\tau \varphi} |u|^2+
\tau\gamma \varphi  e^{2\tau\varphi} |\nabla  u|^2 \right) dS 
\end{split}
\eeq
for all $\tau \ge \tau_*$,   $\gamma \ge \gamma_{*}$ and     $u\in H^2(\Omega)$.
}
\end{lemma}
%The proof of Lemma \ref{lemma:carleman} has been contained in Appendix. 
{
\begin{remark} Obviously, $\psi(x) = e^{x_1}$ satisfies the assumption of Lemma \ref{lemma:carleman}. For our applications, we   only make use of a simpler version of Lemma \ref{lemma:carleman}. More precisely, by fixing a $\gamma>\gamma_*$, we may drop $\gamma$ out of \eqref{eq:Carleman}  and conclude  that the estimate
$$
  \int_{U}e^{2\tau\varphi}(\tau^3  |u|^2+\tau   |\nabla u|^2)dx 
 \lesssim   \int_{\p U}e^{2\tau\varphi}\left(  |\nabla  u|^2  +\tau^3   |u|^2\right) dS 
+\int_{U} e^{2\tau\varphi} |L_\alpha u|^2dx 
$$
holds true for all $\tau \ge \tau_*$ and $u \in H^2(U)$. This estimate now is of the same form as the ones from \cite{Isakov}.  \end{remark}}

%  Similar Carleman estimates hold for a large class of second oder partial differential operators, which even are  not necessarily elliptic, provided that the weighted function $\psi$
%belongs to $C^2(\overline{U})$ and satisfies so-called pseudo-convex property.  If second oder partial differential operator is elliptic, then the weighted function always fulfills the pseudo-convexity when $\lambda$ is large enough (see. e.g. \cite{Isakov}).  Since we could not find a precise reference covering our situation, 
%For the sake of completeness, we include a full proof here for the sake of completeness.

\begin{proof} 

Let $u \in H^2(U)$, $f:= L_\alpha u$  and $v:=e^{\tau\varphi}u$. We define the conjugated operator 
\beq
W_\varphi v:=e^{\tau\varphi}L_\alpha (e^{-\tau\varphi} v) =  -  e^{\tau\varphi}  \nabla\cdot(\alpha \nabla (e^{-\tau\varphi} v) ).
\eeq
From the expansion of $W_\varphi$, we further introduce 
\begin{align*}
W_2v:&=-\nabla\cdot (\alpha\nabla v)-\tau^2\gamma^2\alpha \varphi^2 | \nabla \psi |^2v, \\
W_1v:&= 2\tau\gamma \alpha\varphi \nabla \psi \cdot\nabla  v+2\tau \alpha \gamma^2\varphi |\nabla \psi|^2 v.
\end{align*}
Straightforward computations yield 
\beq\label{eq:ll}
e^{\tau \varphi}f=W_\varphi v=W_2v+W_1v-\tau\gamma \varphi  \nabla \alpha\cdot \nabla \psi  v-\tau \alpha \gamma^2\varphi |\nabla \psi|^2 v
+\gamma\tau\alpha \varphi \Delta \psi v,
\eeq
where we have use   $\varphi=e^{\gamma \psi}$ and the  identity
$
\Delta \varphi=\gamma^2\varphi |\nabla \psi|^2 +\gamma \varphi \Delta \psi. 
$
By  \eqref{eq:ll},  
$$
\|\frac{1}{\sqrt{\alpha}} W_2 v\|^2+\|\frac{1}{\sqrt{\alpha}} W_1 v\|^2+2\Re (W_2 v,\alpha^{-1}W_1v)_U=\|{g}\|^2_U,
$$
with 
\beq\label{def:g}
{g}:=\frac{1}{\sqrt{\alpha}}(e^{\tau \varphi}f +\tau\gamma \varphi  \nabla \alpha\cdot \nabla \psi  v +\tau \alpha \gamma^2\varphi |\nabla \psi|^2 v
-\gamma \tau \alpha \varphi \Delta \psi v 
). 
\eeq
Then, by the definition of $\varphi$, it follows that    
\beq\label{core}
2\Re (W_2 v,\alpha^{-1}W_1v)_U\leq \|g\|^2_{U}.
\eeq
Our goal now is to establish a  proper lower estimate for the left-hand side of \eqref{core}. To this end, let us denote by $I_{ij}$ the real part of the scalar product between  the $i$-th term of $W_2$ and the $j$-th term of $\alpha^{-1}W_1$.

\no {\bf Term $I_{11}$:}
Integration by parts yields 
\begin{align*}
I_{11}:&=-2\tau\gamma \Re  (\nabla\cdot (\alpha\nabla v),  \varphi \nabla\psi \cdot\nabla  v)_U \\
&=2\tau\gamma\Re \int_U \alpha \nabla v \cdot \nabla (\varphi\nabla \psi \cdot \nabla \overline{v}) dx 
-2\tau\gamma \Re \int_{\p\Omega}\alpha\varphi \p_{n}v   (\nabla \psi \cdot \nabla 
{\overline{v}}) dS \\
&=2\tau\gamma \int_{U} \alpha\varphi \psi''(\nabla v,\nabla\overline{ v}) dx+2\tau\gamma^2 \int_U 
\alpha \varphi |\nabla v\cdot \nabla \psi |^2 dx
+
\tau\gamma 
 \int_U \alpha \varphi  \nabla \psi \cdot\nabla |\nabla v|^2 dx \\
 &\quad
 -2\tau\gamma \Re \int_{\p U}\alpha\varphi \p_{n}v   (\nabla \psi \cdot \nabla 
{\overline{v}}) dS,\\
&=2\tau\gamma \int_{U} \alpha\varphi \psi''(\nabla v,\nabla\overline{ v}) dx+2\tau\gamma^2 \int_U 
\alpha \varphi |\nabla v\cdot \nabla \psi |^2 dx
-
\tau\gamma 
 \int_U \nabla\cdot(\alpha \varphi  \nabla \psi)  |\nabla v|^2 dx 
 \\
 &\quad+\tau\gamma 
 \int_{\p U} \alpha \varphi  \p_n \psi  |\nabla v|^2 dx 
 -2\tau\gamma \Re \int_{\p U}\alpha\varphi \p_{n}v   (\nabla \psi \cdot \nabla 
{\overline{v}}) dS,
\end{align*}  
where $\psi''(\nabla v,\nabla \overline{v} ):=\sum_{1\leq i,j\leq d} \p_{i,j}\psi\p_i v\p_j \overline{v}$, which is obviously real-valued. As the third term of the above sum can be rewritten as  
\begin{align*}
 -
\tau\gamma 
 \int_U \nabla\cdot(\alpha \varphi  \nabla \psi)  |\nabla v|^2 dx 
 =-\tau\gamma\int_{U}(\nabla \alpha\cdot { \nabla \psi} 
 +\alpha \gamma |\nabla \psi|^2+\alpha \Delta \psi ) \varphi |\nabla v|^2 dx, 
\end{align*}
we obtain that 
\begin{align}\label{eq:I11}
I_{11}&=2\tau\gamma \int_{U} \alpha\varphi \psi''(\nabla v,\nabla\overline{ v}) dx+2\tau\gamma^2 \int_U 
\alpha \varphi |\nabla v\cdot \nabla \psi |^2 dx\notag
\\
&\quad
-\tau\gamma\int_{U}(\nabla \alpha\cdot { \nabla \psi} 
 +\alpha \gamma |\nabla \psi|^2+\alpha \Delta \psi ) \varphi |\nabla v|^2 dx\\
 &\quad + \tau\gamma 
 \int_{\p U} \alpha \varphi  \p_n \psi  |\nabla v|^2 dx 
 -2\tau\gamma \Re \int_{\p U}\alpha\varphi \p_{n}v   (\nabla \psi \cdot \nabla 
{\overline{v}}) dS.\notag
\end{align}

\no {\bf Term $I_{12}$:} Similarly,  
\begin{align*}
I_{12}&= -2\tau\gamma^2 \Re (\nabla\cdot(\alpha\nabla v),\varphi|\nabla \psi|^2   v)_U\\
&=2\tau\gamma^2\Re \int_U \alpha\nabla v\cdot \nabla (\varphi|\nabla \psi|^2  \overline{v}) dx -2\tau\gamma^2\Re\int_{\p U} \alpha\varphi|\nabla \psi|^2 \p_n v \overline{v} dS\\
&=2\tau\gamma^2\int_U \alpha\varphi |\nabla \psi|^2   |\nabla v|^2 dx+2\tau\gamma^2\Re
\int_U \alpha(\nabla(\varphi |\nabla \psi|^2) \cdot\nabla v) \overline{v}dx\\
&\quad-2\tau\gamma^2\Re\int_{\p U} \alpha\varphi|\nabla \psi|^2 \p_n v \overline{v} dS.
\end{align*}
{\bf Term $I_{21}$:} Also, Integration by parts results in
\begin{align*}
I_{21}&=-2\tau^3\gamma^3 \Re( \alpha |\nabla \psi |^2\varphi^2 v, \varphi \nabla \psi \cdot \nabla v  )_U = -2\tau^3\gamma^3  \int_{U} \alpha |\nabla\psi |^2 \varphi^3 \nabla\psi \cdot \nabla |v|^2 dx \\
&= 2\tau^3\gamma^3\int_U\nabla\cdot( \alpha |\nabla\psi |^2 \varphi^3 \nabla\psi) |v|^2 dx  
 -2\tau^3\gamma^3  \int_{\p U} \alpha |\nabla\psi |^2 \varphi^3 \p_n \psi  |v|^2 dx.
\end{align*}
{\bf Term $I_{22}$:} A direct computation implies 
\begin{align*}
I_{22}&=-2\tau^3\gamma^4 \int_U \alpha|\nabla \psi|^4\varphi^3 |v|^2 dx.  
\end{align*}
Regrouping the terms of all expansions of $I_{ij}$ and  as the second term in the right-hand side of  \eqref{eq:I11} is nonnegative, we have 
\begin{align}\label{sum0}
2\Re (W_2 v,\alpha^{-1} W_1 v)\geq \int_{U} \tau^3\gamma^4 \alpha_0|v|^2 dx+\int_{U} \tau\gamma^2 \alpha_1|\nabla v|^2 dx +T_1+T_2,
\end{align}
with
\begin{align*}
\displaystyle \alpha_0:=& \frac{2}{\gamma}\nabla\cdot( \alpha |\nabla\psi |^2 \varphi^3 \nabla\psi) -2 \alpha |\nabla \psi|^4\varphi^3, \quad \alpha_1:=  (\alpha |\nabla \psi|^2  -\frac{1}{\gamma}\nabla \alpha\cdot {\nabla \psi }
 -\frac{\alpha}{\gamma}\Delta \psi) \varphi, \\
T_1:=&2\tau\gamma \int_{U} \alpha\varphi \psi''(\nabla v,\nabla\overline{ v}) dx+2\tau\gamma^2\Re
\int_U \alpha(\nabla(\varphi |\nabla \psi|^2) \cdot\nabla v) \overline{v}dx,\\
T_2:=& \tau\gamma 
 \int_{\p U} \alpha \varphi  \p_n \psi  |\nabla v|^2 dx 
 -2\tau\gamma \Re \int_{\p U}\alpha\varphi \p_{n}v   (\nabla \psi \cdot \nabla 
{\overline{v}}) dS-2\tau\gamma^2\Re\int_{\p U} \alpha\varphi|\nabla \psi|^2 \p_n v \overline{v} dS\\  &-2\tau^3\gamma^3  \int_{\p U} \alpha |\nabla\psi |^2 \varphi^3 \p_n \psi  |v|^2 dx.
\end{align*}
{
Let us now verify the following estimates:
\beq\label{eq:claim}
 \varphi^3 \lesssim      \alpha_0   \textrm{ a.e. in } U  \quad\text{and}\quad   \varphi   \lesssim   \alpha_1 \textrm{ a.e. in } U,
\eeq
for all sufficiently large $\gamma>0$. Indeed,    since $\alpha(x) \ge \alpha_{\min} >0$ and  $|\nabla \psi(x)|\geq c>0$ hold true for all $x \in \overline U$,  and  $D^2\psi\in L^\infty(U)^{d\times d}$, we find that
$$
\alpha_1 \ge  (\alpha_{\min} c^2  -\frac{1}{\gamma}\nabla \alpha\cdot  \nabla \psi 
 -\frac{\alpha}{\gamma}\Delta \psi)  \varphi \ge  \frac{\alpha_{\min} c^2}{2}  \varphi \textrm{ a.e. in } U
$$
holds for all sufficiently large $\gamma>0$. Furthermore, by virtue of $\varphi=e^{\gamma \psi}$, straightforward computations  imply that
\begin{align*}
\alpha_0=4  \alpha  |\nabla \psi|^4\varphi^3    
-\frac{2}{\gamma }\nabla\cdot(\alpha |\nabla \psi|^2\nabla \psi )\varphi^3   \geq  \alpha_{\min}c^4\varphi^3    \textrm{ a.e. in } U
\end{align*}
holds for all sufficiently large $\gamma>0$. In conclusion, \eqref{eq:claim} is valid. Let us now derive upper estimates for $|T_1|$ and $|T_2|$. Again, by the definition  $\varphi=e^{\gamma \psi}$,  
$$
\nabla (\varphi |\nabla \psi|^2) =  \nabla \varphi  |\nabla \psi|^2  + \varphi \nabla   |\nabla \psi|^2  =  \varphi \left(\gamma \nabla \psi  |\nabla \psi|^2  +  \nabla   |\nabla \psi|^2 \right).
$$
Thus,   for  all $\gamma\ge1$, Young's inequality implies  for every $\varepsilon>0$ that 
\begin{align}\label{eq:T1}
|T_1| \lesssim &  \tau\gamma\int_U \varphi|\nabla v|^2 dx 
+  \tau\gamma^3\int_U  \varphi  |\nabla v|  |v| dx 
\notag\\
\lesssim & \tau\gamma\int_U \varphi|\nabla v|^2 dx 
+\varepsilon \tau\gamma^2 \int_U\varphi |\nabla v|^2 dx
+ \frac{  \tau \gamma^4}{\epsilon} \int_U\varphi |v|^2 dx.
\end{align}
On the other hand, 
in view of   $\psi\in {C^{1,1}(\overline{U})}$, we infer  that 
\begin{align*}
|T_2| \lesssim  (\tau\gamma\int_{\p U}\varphi |\nabla v|^2 dS
+\tau^3\gamma^3\int_{\p U}\varphi^3 |v|^2 dS)
+  \tau \gamma^2\int_{\p U} |\nabla v||v|\varphi dS.
\end{align*}
  Then,   since   $\varphi\geq 1$ on $\p U$,  {Young's inequality} implies for all $\tau \ge1$ that
\begin{align}\label{eq:T2}
|T_2| \lesssim  \tau\gamma\int_{\p U}\varphi |\nabla v|^2 dS
+\tau^3\gamma^3\int_{\p U}\varphi^3 |v|^2 dS.
\end{align}
}
Now, according to \eqref{def:g}, 
\beq\label{eq:g}
\|g\|^2_U\lesssim \int_{U} e^{2\tau \varphi}|f|^2 dx+ \int_U (\tau^2\gamma^2+\tau^2\gamma^4)\varphi^2|v|^2 dx. 
\eeq
It follows therefore from \eqref{core}, \eqref{sum0}, \eqref{eq:claim},  \eqref{eq:T1} with a sufficiently small $\epsilon>0$ and \eqref{eq:g} that
$$
\int_{U}( \tau^3\gamma^4\varphi^3 |v|^2  + \tau\gamma^2\varphi |\nabla v|^2 )dx \\
{\lesssim } \int_U e^{2\tau \varphi}|f|^2 dx+\int_{\p U}
(\tau\gamma\varphi |\nabla v|^2+\tau^3\gamma^3\varphi^3|v|^2 )dS  
$$
{holds true for all sufficiently large $\gamma,\tau \ge 1$}. Substituting the identities  
\beq\label{eq:identities}
v = e^{\tau \varphi} u \quad \textrm{and} \quad \nabla v=e^{\tau\varphi}\nabla u+\tau\gamma  {\varphi e^{\tau\varphi}  u }\nabla \psi 
\eeq
into the inequality above, we obtain that 
$$
\int_{U}( \tau^3\gamma^4\varphi^3 |v|^2   +  \tau\gamma^2\varphi |\nabla v|^2  )dx \\
{ \lesssim } \int_U e^{2\tau \varphi}|f|^2 dx+\int_{\p U}
(\tau\gamma\varphi e^{2\tau\varphi} |\nabla u|^2+\tau^3\gamma^3\varphi^3e^{2\tau \varphi}|u|^2 )dS,
$$
{for all sufficiently large $\gamma,\tau \ge 1$}. 
In addition, \eqref{eq:identities} also yields
$$
\int_{U}( \tau^3\gamma^4\varphi^3 e^{2\tau \varphi} |u|^2 + \tau\gamma^2\varphi e^{2\tau \varphi} |\nabla u|^2)dx
{ \lesssim }   \int_{U}(\tau^3\gamma^4\varphi^3 e^{2\tau \varphi} |v|^2 + \tau\gamma^2\varphi |\nabla v|^2  )dx.
$$
Combining these two inequalities, we finally come to the conclusion that   \eqref{eq:Carleman} is valid. 
\end{proof}
By virtue of Lemma \ref{lemma:carleman},  the same arguments as in \cite[Propositions 2.2-2.4]{Bourgeois} yield the following result:

\begin{proposition}\label{pro:small}{}

\begin{enumerate}
\item[]
\item[\textup{($a$)}]  
Let $\omega_1,\omega_2$ be two domains such that $\omega_1\Subset \Omega$ and $\omega_2\Subset \Omega$. Then, there exist $s,c,\epsilon_0>0$ such that for all  $\epsilon\in(0,\epsilon_0]$ and $u\in H^2(\Omega)$, 
\beq\label{estimate0}
\|u\|_{1,\omega_1}\leq \f{c}{\epsilon}
(\|L_\alpha u\|_{\Omega}+\|u\|_{1,\omega_2})+\epsilon^s\|u\|_{1,\Omega}.
\eeq

\item[\textup{($b$)}] Let $x_0\in \Gamma_a$. Then, there exist a neigborhood $\omega_0$ of $x_0$, and positive constants $s,c,\epsilon_0>0$ such that for  all $\epsilon\in(0,\epsilon_0]$ and $u\in H^2(\Omega)$, 
\beq\label{estimate1}
\|u\|_{1,\Omega\cap \omega_0}\leq \f{c}{\epsilon}
( \|L_\alpha u\|_{\Omega}+\|u\|_{1,\Gamma_a}+\|\p_n u\|_{\Gamma_a})+\epsilon^s\|u\|_{1,\Omega}.
\eeq

\item[\textup{($c$)}]   Let $x^*\in \p\Omega$. Then, there exist a neigborhood $\omega$ of $x^*$ and an open domain $\omega_1\Subset \Omega$ such that   
for each $\kappa\in (0,1)$, there exist $c,\epsilon_0>0$ such that for all $\epsilon\in(0,\epsilon_0]$ and $u\in H^2(\Omega)$,
\beq\label{eq:log:1}
\|u\|_{1,\Omega\cap \omega} \leq e^{c/\epsilon}
(\|L_\alpha u\|_{\Omega}+\|u\|_{1,\omega_1})+\epsilon^\kappa\|u\|_{2,\Omega}.
\eeq

\end{enumerate}

\end{proposition}

For the upcoming results, we shall also make use of the following auxiliary lemma:

\begin{lemma}[{\cite[Lemma 2.3 and Corollary 2.1]{Bourgeois}}]\label{lemma:ineq}

\begin{enumerate}
\item[]
\item[(i)] Let $s$, $\beta$, $A$ and $B$ denote four non-negative real numbers such that $\beta\leq B$. If there exist 
$c,\epsilon_0>0$ such that for all $\epsilon\in (0,\epsilon_0]$, it holds
$$
\beta\leq \frac{c}{\epsilon}A+\epsilon^s B,
$$ 
then there exists $C$,  only depending on $s$ and $c$, such that 
$$
\beta\leq C A^{\f{s}{1+s}} B^{\f{1}{1+s}}. 
$$
\item[(ii)]  Let  $\beta$, $\delta$, $M$ denote three non-negative numbers such that $\beta \leq M$ and $\delta\leq C_0 M$ with some constant $C_0>0$.  If there exit $c,\epsilon_0,\kappa>0$ such that  for all $\epsilon\in (0,\epsilon_0]$, it holds 
$$
\beta \leq e^{c/\epsilon}\delta + \epsilon^\kappa M,
$$ 
then there exists $C$,  only depending on $c,\epsilon_0,$ and $\kappa$, such that 
$$
\beta\leq  C\frac{M}{\log(\frac{C_0 M}{\delta})^\kappa}.
$$
\end{enumerate}
\end{lemma}
By virtue of Lemma \ref{lemma:ineq}, the estimate \eqref{estimate0} in Proposition  \ref{pro:small} (a)    
implies 
\beq\label{est0Holder}
\|u\|_{1,\omega_1}\lesssim (\|L_\alpha u\|_{\Omega}+\|u\|_{1,\omega_2})^{\frac{s}{1+s}} \|u\|_{1,\Omega}^{\f{1}{1+s}} \quad \forall\, u\in H^2(\Omega).
\eeq
Furthermore,  the estimate \eqref{eq:log:1} in Proposition  \ref{pro:small} (c) yields   
\beq\label{est1Holder}
\|u\|_{1,\Omega\cap \omega }\lesssim \frac{\|u\|_{2,\Omega}}{\log(\frac{C_0\|u\|_{2,\Omega}}{\|L_\alpha u\|_{\Omega}+\|u\|_{1,\omega_1}})^{\kappa}}
\quad \forall\, u\in H^2(\Omega), 
\eeq
for some $C_0>0$. 

 To obtain our final result, we need to reduce the $H^1$-regularity term of 
Dirichlet data  in Proposition \ref{pro:small} (b) to a term of $L^2$-regularity { by possibly enlarging the   term $\epsilon^s\|u\|_{1,\Omega}$}. 

\begin{proposition}\label{pro:boundary-boundary}
Let $x_0\in \Gamma_a$. Then, there exist  a neigborhood $\omega_0$ of $x_0$, and positive constants $s,c,\epsilon_0>0$ such that for all $\epsilon\in(0,\epsilon_0]$ and $u\in H^2(\Omega)$, 
\beq\label{estimate2}
\|u\|_{1,\Omega\cap \omega_0}\leq \f{c}{\epsilon}
( \|L_\alpha u\|_{\Omega}+\|u\|_{\Gamma_a}+\|\p_n u\|_{\Gamma_a})+\epsilon^s\|u\|_{2,\Omega}.
\eeq

\end{proposition}

\begin{proof}
Since $\Gamma_a$ itself is a compact $C^{1,1}$-manifold, we have the following well-known interpolation result 
$$
\|u\|_{1,{\Gamma_a}}\lesssim \|u\|_{3/2,\Gamma_a}^{2/3}\|u\|_{0,\Gamma_a}^{1/3} \quad \forall \, u\in H^{3/2}(\Gamma_a)
$$
(see \eqref{comp:interp} and e.g. \cite[Theorem 7.7]{Lions} ). 
Then, Young's inequality yields  for all $\epsilon>0$ that
$$
\|u\|_{1,\Gamma_a} { \lesssim}\epsilon^{3/2} \|u\|_{3/2,\Gamma_a}
+ { \epsilon^{- 3}}\|u\|_{\Gamma_a}. 
$$
Applying this inequality to Proposition \ref{pro:small} (b), we find constants $s',c',\epsilon'_0>0$ such that for all $\epsilon\in(0,\epsilon_0]$ and $u\in H^2(\Omega)$, it holds 
$$
\|u\|_{1,\Omega\cap \omega_0}\leq \f{c'}{\epsilon}
( \|L_\alpha u\|_{\Omega}+\|\p_n u\|_{\Gamma_a})+
\f{c'}{\epsilon^4}
\|u\|_{\Gamma_a}+ {c'}\epsilon^{1/2}\|u\|_{3/2,\Gamma_a}+
\epsilon^{s'}\|u\|_{1,\Omega}.
$$
Setting $s'':=\min\{1/2,s'\}$ and by the embedding result $\|u\|_{3/2,\Gamma_a}\lesssim \|u\|_{2,\Omega}$, there exists a further constant  $c''>0$ such that 
$$
\|u\|_{1,\Omega\cap \omega_0}\leq \f{c''}{\epsilon^{4}}
( \|L_\alpha u\|_{\Omega}+\|\p_n u\|_{\Gamma_a}+ 
\|u\|_{\Gamma_a})+c''\epsilon^{s''}
\|u\|_{2,\Omega},
$$
for all $0<\epsilon\leq \min\{1,\epsilon_0\}$.
By a reparametrization of $\epsilon$, i.e., replacing $\epsilon$ by {$ \hat c\epsilon^{\f{1}{4}}$} for a sufficiently large $\hat c>0$, we may find positive    constants $c,s,  \epsilon_0>0$ such that \eqref{eq:log:1} is valid. 
\end{proof}
Employing  the developed  local results,  we are now in the position to prove  global estimates.
Roughly speaking, Proposition \ref{pro:boundary-boundary}  enables us to ``transfer'' Cauchy data on $\Gamma_a$ to 
a neighborhood $\omega_0$ of every  point $x_0\in \Gamma_a$, in particular to a subdomain $\omega_2$ such that $\omega_2\Subset \omega_0$.
 Proposition \ref{pro:small} (a)  allows us to ``transfer'' data from this open domain $\omega_2$ to another small domain  $\omega_1\Subset \Omega$, in particular, to the one ``near'' to a point $x_0\in \Gamma_i$. Lastly, Proposition \ref{pro:small} (c)  ``transfer'' data on an open domain $\omega_1\Subset \Omega$ to  a neighborhood $\omega$  of  
$x_*\in \Gamma_i$ (See Fig. 2 below).  In the sequel we explain what exactly does ``transfer'' mean. 
%The analogue of Theorem \ref{the:H1:log} is presented in \cite{Bourgeois}  without proof. For sake of completeness, we give a full proof here.

\begin{theorem}\label{the:H1:log}
For every $\kappa\in (0,1)$, there exist $c>0$  and $\epsilon_0$ such that 
\beq\label{eq:final}
\|u\|_{1,\Omega}\leq e^{c/\epsilon}
( \|L_\alpha u\|_{\Omega}+\|u\|_{\Gamma_a}+\|\p_n u\|_{\Gamma_a})+\epsilon^\kappa \|u\|_{2,\Omega}
\eeq
holds true for all $\epsilon\in(0,\epsilon_0]$ and $u\in H^2(\Omega)$. 
\end{theorem}

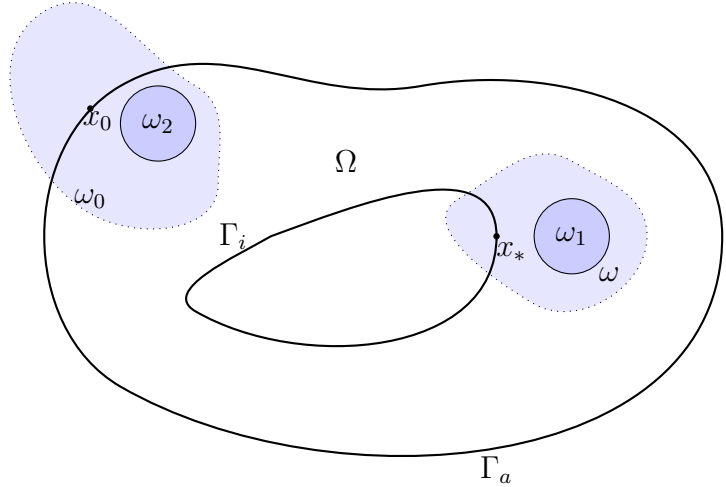
\begin{figure}[h!]\label{Fig:proof}
  \begin{center}
  \begin{tikzpicture}
   \filldraw[fill=blue!10] [dotted](-3.0,3.0)%draw and fill omega0 
                     to [out=30,in=150] (-1.0,2.0) %out means the angle of in curve from the previous 
                      to [out=-30,in=90] (-0.7,0.7) 
               to [out=-90,in=-30] (-2.4,0.3)
             to  [out=150,in=-150]  (-3.0,3.0);
\filldraw[fill=blue!10] [dotted](2.5,0.5)%draw and fill omega  
                     to [out=30,in=150] (4,1) %out means the angle of in curve from the previous 
                      to [out=-30,in=90] (5,0)
                            to [out=-90,in=-30]  (3.5,-0.9)
             to  [out=150,in=-150]  (2.5,0.5);
      \filldraw  (4.0,0) [fill=blue!20]  circle (0.5cm); %omega1      
     \filldraw (-1.5,1.5) [fill=blue!20]  circle (0.5cm);%omega2
   \draw [thick] (-2,2) % Draws a line
      to [out=30,in=190] (2,2) % out, put is the angle control
      to [out=10,in=90] (6,0) 
      to [out=-90,in=-30] (-2,-2)
      to  [out=150,in=-150] (-2,2);
      \draw [thick] (0,0) % Draws a line
      to [out=30,in=180] (0,0) % out, put is the angle control
      to [out=20,in=90] (3,0) 
      to [out=-90,in=-30] (-1,-1)
      to  [out=150,in=-150] (0,0);
       \draw (1,1) node {$\Omega$};% add text and symbols
            \draw (-0.5,0) node {$\Gamma_i$};  
            \draw (3,-3.1) node {$\Gamma_a$};
%            \draw[dotted]   (3.0,0) circle (2.0cm);
            \filldraw (3.0,0) circle(1pt); 
            \draw (3.2,-0.2) node {$x_*$}; %xstar
%              \draw[dotted] (-2.4,1.7) circle (2.1cm);
               \draw (-2.3,1.54) node {$x_0$}; %x0
               \filldraw (-2.4,1.7) circle (1pt); %x0
                    \draw (-1.5,1.5) node {$\omega_2$};
                        \draw (-2.4,0.5) node {$\omega_0$};
                        \draw (4,0) node {$\omega_1$};
                         \draw (4.5,-0.5) node {$\omega$};
%                       \draw [->] (-2.4,1.7) -- (-3.81,3.14);                          \draw  (-3.1,2.25)  node {$r$}; 

                           \end{tikzpicture}
  \caption{Proof of  Theorem \ref{the:H1:log}}
  \end{center}
\end{figure}

\begin{proof} 
Let $\kappa\in (0,1)$. We first study  the local boundary estimate over $\Gamma_i$. Let  $x_*\in \Gamma_i$ be arbitrarily fixed.  Proposition \ref{pro:small}  (c) implies  the existence of a neigborhood $\omega$ of $x^*$, an open domain $\omega_1\Subset \Omega$  and positive constants  $c',\epsilon'_0>0$  such that
\beq\label{eq:log00}
\|u\|_{1,\Omega\cap \omega} \leq e^{c/\epsilon}
(\|L_\alpha u\|_{\Omega}+\|u\|_{1,\omega_1})+\epsilon^\kappa\|u\|_{2,\Omega}
\eeq
holds for all $\epsilon\in(0,\epsilon'_0]$ and $u\in H^2(\Omega)$. Next, let    $x_0\in \Gamma_a$ be arbitrarily fixed. Proposition 
\ref{pro:boundary-boundary}   and Lemma \ref{lemma:ineq} (i) imply the existence of  a neighborhood $\omega_0$  of $x_0$ and a constant
$\theta'\in (0,1)$ such that  
\beq\label{eq:Holder01}
\|u\|_{1,\Omega\cap \omega_0}\lesssim (\|L_\alpha u\|_\Omega+\|u\|_{\Gamma_a}+\|\p_n u\|_{\Gamma_a})^{\theta'}\|u\|_{2,\Omega}^{1-\theta'} \quad \forall\, u\in H^2(\Omega).  
\eeq   
{
Our goal now is to prove  that the estimation  \eqref{eq:log00} remains true with $\|u\|_{1,\omega_1}$ replaced by the Dirichlet data on $\Gamma_a$. To this aim, let us 
 select a domain $\omega_2$ such that $\omega_2\Subset \omega_0\cap \Omega$ (see Fig. 2).  }
By 
Proposition \ref{pro:small} (a) and Lemma \ref{lemma:ineq} (i),  it holds for some $\theta\in (0,1)$  that 
\beq\label{eq:Holder00}
\|u\|_{1,\omega_1}\lesssim (\|L_\alpha u\|_\Omega+\|u\|_{1,\omega_2})^{\theta}\|u\|_{1,\Omega}^{1-\theta} \quad \forall\, u\in H^2(\Omega). 
\eeq
  Since $\omega_2\Subset  \Omega\cap \omega_0$, it follows from   \eqref{eq:Holder01}-\eqref{eq:Holder00}  and $\|L_\alpha u\|_{\Omega}\lesssim \|u\|_{2,\Omega}$
that 
\begin{align*}
\|u\|_{1,\omega_1}\lesssim& (\| L_\alpha u\|_\Omega+(\|L_\alpha u\|_\Omega+\|u\|_{\Gamma_a}+\|\p_n u\|_{\Gamma_a})^{\theta'}\|u\|_{2,\Omega}^{1-\theta'} )^{\theta}\|u\|_{1,\Omega}^{1-\theta}\\
\lesssim& (\|L_\alpha u\|_\Omega^{\theta'}\|u\|_{2,\Omega}^{1-\theta'}+(\|L_\alpha u\|_\Omega+\|u\|_{\Gamma_a}+\|\p_n u\|_{\Gamma_a})^{\theta'}\|u\|_{2,\Omega}^{1-\theta'} )^{\theta}\|u\|_{1,\Omega}^{1-\theta} \quad \forall\, u\in H^2(\Omega).
\end{align*}
Then, by an elementary inequality $x^{\theta'}
+y^{\theta'}\leq 2^{1-\theta'}(x+y)^{\theta'}$ for all $x,y\geq 0$,    we obtain
$$
\|u\|_{1,\omega_1}\lesssim (
\|L_\alpha u\|_\Omega+\|u\|_{\Gamma_a}+\|\p_n u\|_{\Gamma_a})^{\theta'\theta} \|u\|_{2,\Omega}^{(1-\theta')\theta}\|u\|_{1,\Omega}^{1-\theta}   \quad \forall\, u\in H^2(\Omega),
$$
and hence 
\beq \label{eq:log01}
\|u\|_{1,\omega_1}\lesssim (
\|L_\alpha u\|_\Omega+\|u\|_{\Gamma_a}+\|\p_n u\|_{\Gamma_a})^{\theta'\theta} \|u\|_{2,\Omega}^{1-\theta'\theta} \quad \forall\, u\in H^2(\Omega). 
\eeq
Combining \eqref{eq:log01} and \eqref{eq:log00} together,  we obtain for all $\epsilon\in(0,\epsilon'_0]$ and $u\in H^2(\Omega)$  that 
$$
\|u\|_{1,\Omega\cap \omega} \leq e^{c'/\epsilon}
(\|L_\alpha u\|_{\Omega}+
(\|L_\alpha u\|_\Omega+\|u\|_{\Gamma_a}+\|\p_n u\|_{\Gamma_a})^{s} \|u\|_{2,\Omega}^{1-s} )+\epsilon^\kappa\|u\|_{2,\Omega},
$$ 
with $s=\theta'\theta$, from which it follows that 
\beq\label{final0}
\|u\|_{1,\Omega\cap \omega} \lesssim  e^{c'/\epsilon}
(\|L_\alpha u\|_\Omega+\|u\|_{\Gamma_a}+\|\p_n u\|_{\Gamma_a})^{s} \|u\|_{2,\Omega}^{1-s} +\epsilon^\kappa\|u\|_{2,\Omega}. 
\eeq
Furthermore, Young's inequality implies
$$
e^{c'/\epsilon}
(\|L_\alpha u\|_\Omega+\|u\|_{\Gamma_a}+\|\p_n u\|_{\Gamma_a})^{s} \|u\|_{2,\Omega}^{1-s} \leq 
\frac{e^{\f{c'}{s\epsilon}}}{\epsilon^{\f{\kappa(1-s)}{s}}} (\|L_\alpha u\|_\Omega+\|u\|_{\Gamma_a}+\|\p_n u\|_{\Gamma_a})+\epsilon^\kappa\|u\|_{2,\Omega}.
$$
{Applying this inequality to \eqref{final0} and
  choosing a sufficiently large $c>0$ and a sufficiently small $\epsilon_0>0$,  we obtain the desired estimate}
\beq\label{eq:final2}
\|u\|_{1,\Omega\cap \omega}  \leq e^{c/\epsilon}
( \|L_\alpha u\|_{L^2}+\|u\|_{\Gamma_a}+\|\p_n u\|_{\Gamma_a})+\epsilon^\kappa \|u\|_{2,\Omega},
\eeq
for all $\epsilon\in (0,\epsilon_0]$ and all $u\in H^2(\Omega)$.

  For the case $x_0$ being a point on  $\Gamma_a$, we can readily prove a better estimate.  Indeed,   using  \eqref{eq:Holder01} 
and Young's inequality again, we know that there exist some $c>0$ such that  for all $u\in H^2(\Omega)$,  
$$
\|u\|_{1,\Omega\cap \omega_0}  \lesssim  \frac{c}{\epsilon^{\f{\kappa(1-{\theta'})}{\theta'} }}
( \|L_\alpha u\|_{L^2}+\|u\|_{\Gamma_a}+\|\p_n u\|_{\Gamma_a})+\epsilon^\kappa \|u\|_{2,\Omega}. 
$$ 
This estimate is still true if we replace the term $ \frac{c}{\epsilon^{\f{\kappa(1-\theta')}{\theta'} }}
$ by $e^{\f{c}{\epsilon}}$,  provided $\epsilon$ is sufficiently small enough.   Therefore, the estimate \eqref{eq:final2} holds true with 
$\|u\|_{1,\Omega\cap \omega} $ replaced by $\|u\|_{1,\Omega\cap \omega_0} $.  Of course, we may need to choose another constants $c$ and $\epsilon_0$ if necessarily. In the same manner, for every $\omega' \Subset \Omega$,  one can show that  estimate \eqref{eq:final2} is still valid  with 
$\|u\|_{1,\Omega\cap \omega} $ replaced  by {$\|u\|_{1,\omega'}$}. 
 Patching together all local estimates, we conclude from { the compactness of $\overline \Omega$}  that  \eqref{eq:final} is valid.
\end{proof}
We close this section by proving Theorem \ref{continuity}:
\begin{proof}[Proof of Theorem \ref{continuity}] {Let $u \in H^2(\Omega)$ be a solution of \eqref{eq:homo}  within $\mathcal{M}_M$ and $\kappa\in (0,1)$. }
Since $L_\alpha u=0$,   Theorem  \ref{the:H1:log}  and Lemma \ref{lemma:ineq} (ii)   imply
$$
\|u\|_{1,\Omega}\leq \frac{C\|u\|_{2,\Omega}}{
\log(\frac{C_0 \|u\|_{2,\Omega}}{  \|u\|_{\Gamma_a}+\|\p_n u\|_{\Gamma_a})})^\kappa}
$$
for some positive constants $C$ and $C_0$, independent of $u$.  Since the mapping 
$y\mapsto \frac{y}{(\ln (y/y_0))^\kappa}$ is increasing over 
$(y_0,+\infty)$ and $\|u\|_{2,\Omega} \le M$, it holds that  
\beq\label{eq1}
\|u\|_{H^1(\Omega)}\leq C\frac{M}{
\left(\ln{(C_0\frac{M}{  \|u\|_{\Gamma_a}+\|\p_n u\|_{\Gamma_a}}} )\right)^\kappa}
\eeq
Since   
$\|u\|_{\Gamma_a}\lesssim \|\p _n u\|_{\Gamma_a}$,  
we conclude (by changing $C$ and $C_0$ if necessary) that Theorem \ref{continuity} is valid. 
\end{proof}

\section{Proof of Theorem \ref{the:main}}

To prove  Theorem \ref{the:main}, we apply Lemma \ref{lemma:vsc}  by constructing   suitable orthogonal projections on $L^2(\Gamma_i)$. Our proof is based on  the   conditional stability estimate (Theorem \ref{continuity}) along with  the complex interpolation theory and the  following  Gelfand triple:

\begin{enumerate}

\item[(G1)]   $H^{1/2}(\Gamma_i)\subset L^2(\Gamma_i)\subset H^{-1/2}(\Gamma_i)$ with dense and continuous embeddings;

\item[(G2)] $\{H^{1/2}(\Gamma_i), H^{-1/2}(\Gamma_i)\}$  forms an adjoint pair with the duality product $\langle\cdot,\cdot \rangle_{H^{-1/2}(\Gamma_i),H^{1/2}(\Gamma_i)} $; 

\item[(G3)] the duality pairing $\langle\cdot,\cdot \rangle_{H^{-1/2}(\Gamma_i),H^{1/2}(\Gamma_i)} : H^{-1/2}(\Gamma_i) \times H^{1/2}(\Gamma_i) \to \mathbb C$ satisfies 
$$
\langle v,u\rangle_{H^{-1/2}(\Gamma_i),H^{1/2}(\Gamma_i)}=(v,u)_{L^2(\Gamma_i)} \quad \forall\, u\in  H^{1/2}(\Gamma_i),\,\,v\in L^2(\Gamma_i). 
$$

\end{enumerate}
Since the inner-product $(\cdot,\cdot)_{H^{1/2}(\Gamma_i)}$ is a symmetric sesquilinear  form over $H^{1/2}(\Gamma_i)$,  the operator $\B: H^{1/2}(\Gamma_i)  \to H^{-1/2}(\Gamma_i)$ defined by 
$$
\langle \B  u, v \rangle_{H^{-1/2}(\Gamma_i),H^{1/2}(\Gamma_i)}:= (u,v)_{H^{1/2}(\Gamma_i)} 
\q \forall\, u,v\in  H^{1/2}(\Gamma_i) 
$$
is linear and bounded.  We can then define an unbounded operator  $\A: D(\A)\subset L^2(\Gamma_i)\to L^2(\Gamma_i)$ as follows: 
$$
\A u:=\B u \quad \forall\, u\in D(\A),
$$
with the domain 
$$
D(\A)=\{u\in H^{1/2}(\Gamma_i) \, | \,Ê\B u\in L^2(\Gamma_i) \}. 
$$
One can infer that $\A: D(\A)\subset L^2(\Gamma_i)\to L^2(\Gamma_i)$ is a densely defined and closed operator (cf. \cite[Theorem 1.25]{Yagibook}). Further properties of this operator is summarized in the following lemma:    

\begin{lemma}[{\cite[Chapter 1, Section 8]{Yagibook}}]
The operator $\A:D(\A)\subset L^2(\Gamma_i) \to  L^2(\Gamma_i)$ is   densely defined, closed, self-adjoint and $m$-accretive. Furthermore, it satisfies
\beq\label{eq:inner0}
(\A u,v)_{\Gamma_i}=(u,v)_{H^{1/2}(\Gamma_i)}\quad \forall \, u,v\in  D(\A).
\eeq
\end{lemma}

By (\ref{eq:inner0}), we obtain that 
$$
(\A u,u)_{\Gamma_i}=\|u\|_{1/2,\Gamma_i}^2\geq \|u\|_{0,\Gamma_i}^2 \quad \forall u\in D(\A).
$$
Then, in view of the compactness of the embedding $D(\A)\subset  L^2(\Gamma_i)$, we  infer that there exists a complete orthonormal basis $\{ e_n\}_{n=1}^\infty \subset L^2(\Gamma_i)$ such that 
\beq
(\A u,u)_{\Gamma_i}=\sum_{n=1}^\infty \lambda_n |(u,e_n)_{\Gamma_i}|^2  \quad \forall\, u\in D(\A),
\eeq
where $1\leq\lambda_{1}\leq \lambda_2 \leq \cdots$, $\lim_{n\to\infty} \lambda_n=+\infty$, and, for every $n\in \mathbb{N}^+$,  $e_n $ is the eigenfunction of $\A$ for the eigenvalue of $\lambda_n$, i.e., 
$$
\A e_n=\lambda_n  e_n \quad \forall n\in \mathbb N.
$$
For every $s\in \R$, the fractional power $\A^s$ of $\A$ can be defined as
\beq\label{def:As}
\A^s u:=\sum_{n=1}^\infty \lambda^s_n  (u, e_n)_{\Gamma_i} e_n \q \forall\, u\in D(\A^s),
\eeq
where the domain $D(\A^s)$ is given by 
\beq\label{eq:Bs}
D(\A^s)=\{ u\in L^2(\Gamma_i)\mid \,\, \sum_{n=1}^\infty \lambda_n^{2s} |(u,e_n)_{\Gamma_i}|^2 <\infty\}. 
\eeq
Then, for each $s\geq 0$, $\A^s: D(\A^s)\subset L^2(\Gamma_i) \to   L^2(\Gamma_i)$ is also self-adjoint, and $D(\A^s)$ is a Banach space equipped with the norm 
\beq\label{eq:Asnorm}
\| u \|_{D(\A^s)}:=\|\A^s u\|_{L^2(\Gamma_i)}=\left(\sum_{n=1}^\infty \lambda_n^{2s}|(u,e_n)_{\Gamma_i}|^2\right)^{1/2} \quad \forall u\in D(\A^s), 
\eeq
which is also equivalent to the corresponding graph norm of $(\A^s,D(\A^s))$ (for more details, we refer to \cite{Reed,Wloka}). Let us mention that for all $\theta\in [0,1/2]$, it holds that 
\beq\label{eq:normeq0}
D(\A^{\theta})=[L^2(\Gamma_i), H^{1/2}(\Gamma_i)]_{2\theta}=H^{\theta}(\Gamma_i) 
\eeq
 with norm equivalence. The first identity is from {\cite[Corollary 2.4]{Yagibook}}, while the second  one is due to \eqref{comp:interp}.

\begin{proof}[Proof of Theorem \ref{the:main}]
{
If  $q^\dag\neq 0$, then  \eqref{vscA} holds true for every  concave index function $\Psi$.}  Let therefore $0 \neq q \in H^{s}(\Gamma_i)$ for some $s\in  (0,1/2]$.  Let us introduce  a  family of $\{P_\lambda\}_{\lambda\geq \lambda_1}$ in $L^2(\Gamma_i)$, where $P_\lambda: L^2(\Gamma_i)\to L^2(\Gamma_i)$ is define by  
$$
P_\lambda q=\sum_{\lambda_n\leq \lambda} (q,e_n)_{\Gamma_i} e_n. 
$$
 Since $q^\dag\in H^{s}(\Gamma_i)$, it follows from \eqref{eq:normeq0}  that $q^\dag\in D(\A^s)$. Then, using \eqref{eq:Asnorm}, we have 
\begin{align}\label{decay00}
\|(I-P_\lambda)q^\dag\|^2_{0,\Gamma_i}=&
\sum_{\lambda_n>\lambda}  |(q^\dag,e_n)_{\Gamma_i}|^2\notag\\
\leq& \sum_{\lambda_n>\lambda} \f{\lambda_n^{2s}}{\lambda^{2s}} |(q^\dag,e_n)_{\Gamma_i}|^2 
\leq  
 \f{1}{\lambda^{2s}} \|q^\dag\|_{D(\A^s)}^2. 
\end{align}
Thus, it remain to estimate the inner product $( q^\dag,q^\dag-q)_{0,\Gamma_i}$ for every $ q\in \mathcal U_{q^\dag}$. To this aim, let $ q\in \mathcal U_{q^\dag}$. As $q^\dag\in D(\A^{s})$, there exits $\omega:=\A^s q^\dag \in L^2(\Gamma)$ such $q^\dag=\A^{-s}\omega$. 
From the definition of $\A^{ s}$, it follows that 
\begin{align}\label{innerproduct00}
\Re(q^\dag,P_\lambda(q^\dag -q))_{\Gamma_i}=\Re(\omega,\A^{-s} P_\lambda (q^\dag -q))_{\Gamma_i}. 
\end{align}
Since $\|q^\dag-q\|_{\Gamma_i}
=\sum_{n=1}^\infty  |(q^\dag-q,e_n)_{\Gamma_i}|^2$, we have 

\beq\label{eq:important}
\begin{array}{ll}
\displaystyle \|\A^{- s}P_\lambda (q^\dag-q)\|^2_{0,\Gamma_i}&=\sum_{\lambda_n\leq \lambda}
\lambda^{-2s}_n |(q^\dag-q,e_n)_{\Gamma_i}|^2 
\end{array}
\eeq
Since $H^{1/2}(\Gamma_i)$ is dense in $L^2(\Gamma_i)$ and { by (G3)} , we have 
\beq\label{eq:important1}
|(q^\dag-q,e_n)_{\Gamma_i}|\leq \|q^\dag-q\|_{-\frac{1}{2},\Gamma_i}\|e_{n}\|_{\f1 2,\Gamma_i} 
\lesssim \sqrt{\lambda _n}  \|q^\dag-q\|_{-\frac{1}{2},\Gamma_i},
\eeq
where we have uses   \eqref{eq:Asnorm}.  Combining \eqref{eq:important} and \eqref{eq:important1} yields 
\beq\label{eq:important20}
 \|\A^{- s}P_\lambda (q^\dag-q)\|^2_{0,\Gamma_i}
 \lesssim \lambda^{1-2s}  { \|q^\dag-q\|_{-\frac{1}{2},\Gamma_i}^2}. 
\eeq
According to the definition of $\mathcal{U}_{q^\dag}$,  it holds that  $q-q^\dag\in H^{1/2}(\Gamma_i)$ with $\|q-q^\dag\|_{1/2,\Gamma_i} \le M_0$. Thus, by  a classical elliptic regularity result (cf. \cite{Grisvard,Lions}), there exists a constant $M>0$ such that $u:=S(q)-S(q^\dag)\in \mathcal{M}_M$; see \eqref{eq:M} for the definition of $\mathcal{M}_M$. We note that the constant $M$ may depend on $M_0$, but not on $q \in \mathcal{U}_{q^\dag}$. Furthermore, by definition, $u=S(q)-S(q^\dag) \in H^2(\Omega)$ satisfies    
$$
\begin{aligned}
\langle q^\dag -q, \text{tr}(v)\rangle_{H^{-1/2}(\Gamma_i),H^{1/2}(\Gamma_i)} = \int_{\Gamma_i} (q^\dag -q ) \text{tr}(\overline v)  d S  = \int_{\Omega} \alpha \nabla u \cdot \nabla \overline v \, dx  +\int_{\Gamma_a}k \text{tr}(u)\text{tr}(\overline{v}) dS \\
=  \int_{\Gamma_i}  \text{tr}( \nabla u \cdot n)  \text{tr}(\alpha \overline v) d S  \quad \forall  v \in H^1(\Omega) \textrm{ satisfying }  \text{tr}(  v) = 0 \textrm{ on } \Gamma_a\\
\end{aligned}
$$
 Thus, as  $\text{tr} : H^1 (\Omega) \to H^{1/2}(\partial \Omega)$ is surjective, we deduce  from $\Gamma_a \cap \Gamma_i = \emptyset$,   the boundedness of     $H^1(\Omega)\ni u\mapsto  \p_n u \in   H^{-1/2}(\Gamma_i)$ and the estimate 
 $$
 \|\alpha v\|_{\f{1}{2},\Gamma_i}\lesssim \|\alpha\|^{\f{1}{2}}_{C^1(\Gamma_i)} \|\alpha\|^{\f{1}{2}}_{C(\Gamma_i)} \|v\|_{\f{1}{2},\Gamma_i}\quad \forall\, v\in H^{\f{1}{2}}(\Gamma_i)
 $$ (see e.g. \cite[Page 49]{Yagibook})
 that 
$$
\|q^\dag-q\|_{-\frac{1}{2},\Gamma_i} \lesssim \|\alpha\|_{C^1(\Gamma_i)} \|\partial_n u\|_{-\frac{1}{2},\Gamma_i} \le C_n \|\alpha\|_{C^1(\Gamma_i)}  \|u\|_{1,\Omega},
$$
with $C_n>0$, independent of $u, \, q$ and $q^\dag$.  By this inequality and since $u \in \mathcal{M}_M$ satisfies    \eqref{eq:homo}, Theorem \ref{continuity} and \eqref{eq:important20}  imply  the existence of positive constants $C,C_0>0$, independent of $q$ and $q^\dag$ such that
$$
\|\A^{-s} P_\lambda (q^\dag-q)\|_{0,\Gamma_i} \leq      \lambda^{1/2-s}  \|\alpha\|_{C^1(\Gamma_i)}  C_n  \f{C M}{\log(\frac{C_0M}{\|A(q^\dag)-A(q)\|_{0,\Gamma_a}})^\kappa},
$$
which, together with (\ref{innerproduct00}), implies that 
\beq\label{eq:exp0}
\Re(q^\dag,P_\lambda (q^\dag -q) )_{\Gamma_i}\leq  \lambda^{1/2-s} \|\omega\|_{\Gamma_i}  \|\alpha\|_{C^1(\Gamma_i)}   C_n \f{C M}{\log(\frac{C_0M}{\|A(q^\dag)-A(q)\|_{0,\Gamma_a}})^\kappa}.
\eeq
 To apply Lemma \ref{lemma:vsc}, we first notice that the continuity of the trace operator yields $\|A(q^\dag)-A(q)\|_{\Gamma_a}\leq c'M$ with a constant $c'>0$, independent of $q^\dag$ and $q$.  
Without loss of generality, we assume that  $C_0$ is large enough so that $\frac{C_0 }{c'}>e^{\kappa+1} $. Otherwise, we can enlarge $C_0$ and $C$ at the same time since the mapping
$M\mapsto \frac{M}{\log(\frac{M}{\delta})^\kappa} $ is increasing over $[\delta,+\infty)$. It is readily checked that  the function 
$$
G_0(\delta)= { \|\omega\|_{\Gamma_i}  \|\alpha\|_{C^1(\Gamma_i)} }\f{C M}{\log(\frac{C_0M}{\delta})^\kappa }\quad \forall \, \delta\in (0,c'M], 
$$
is concave, continuous and strictly increasing over $(0,c'M]$.  In conclusion, the mapping  
$$
\Psi_0 : (0,\infty) \to (0,\infty), \quad  \Psi_0(\delta) := \left \{
\begin{split}
   & (d_\delta G)(c'M)(x-c'M)+G(c'M)
     &\textrm{ if } \delta \in (c'M,\infty) \\
& G_0(\delta)  &\textrm{ if } \delta \in (0,c'M] 
 \end{split}
 \right.
$$
 is a concave  index function   satisfying 
\beq\label{eq:exp20}
\Re(q^\dag,P_\lambda (q^\dag -q) )_{\Gamma_i}\leq  \lambda^{1/2-s} \Psi_0(\|A(q^\dag)-A(q)\|_{0,\Gamma_a}) \quad \forall q\in \mathcal{U}_{q^\dag}.
\eeq
In view of (\ref{decay00}) and (\ref{eq:exp20}),   if $0< s < 1/2$, then Lemma \ref{lemma:vsc} is applicable and yields the assertion.  If $s=1/2$, then \eqref{eq:exp20} implies 
$$
\Re(q^\dag, P_\lambda(q^\dag -q) )_{\Gamma_i}\leq  \Psi_0(\|A(q^\dag)-A(q)\|_{0,\Gamma_a}) \quad \forall q\in \mathcal{U}_{q^\dag}.
$$
By passing $\lambda \to\infty$,  we see that  the claim  is true for $\Psi=\Psi_0$.
\end{proof}

\end{document}